\documentclass[10pt]{amsart}
\usepackage[centertags]{amsmath}
\usepackage{amsfonts,mathrsfs}
\usepackage{newlfont}
\usepackage[all]{xy}
\usepackage{graphicx}
\usepackage[latin1]{inputenc}
\usepackage[english]{babel}
\usepackage[T1]{fontenc}
\usepackage{amscd}
\usepackage{amsthm}
\usepackage{amssymb}
\usepackage{verbatim}
\usepackage[a4paper,top=3.9cm,bottom=3.4cm,left=2.8cm,right=2.8cm]{geometry}
\newlength{\defbaselineskip}
\newcommand{\setlinespacing}[1]%
           {\setlength{\baselineskip}{#1 \defbaselineskip}}

\def\M{\mathfrak{M}}
\def\kis{\text{Mod}^{\phi}_{/W_2[[u]]}}

\newtheorem{thm}{Theorem}[section]
\newtheorem{prop}[thm]{Proposition}

\newtheorem*{thmintro}{Theorem}
\newtheorem*{propintro}{Proposition}

\newtheorem{lem}[thm]{Lemma}
\newtheorem{cor}[thm]{Corollary}
\theoremstyle{definition}
\newtheorem{defn}[thm]{Definition}
\newtheorem{lemdefn}[thm]{Lemma-Definition}

\newtheorem{ex}[thm]{Example}
\theoremstyle{remark}
\newtheorem{rem}[thm]{Remark}

\date{}

\DeclareMathOperator{\Hom}{Hom} \DeclareMathOperator{\Ext}{Ext}

\newcommand \piR {\pi R\setminus\{0\}}

\newcommand \ap{\alpha_p}

\newcommand \bbf[1]{\textbf{\textit{#1}}}

\newcommand \F{\mathbb{F}}

\newcommand \g{\mathcal{G}}
\newcommand \glb{\g^{(\lambda)}}
\newcommand \gmu{\g^{(\mu)}}
\newcommand \gmup{\g^{(\mu^p)}}
\newcommand \glbp{\g^{(\lambda^{p})}}

\newcommand \glx[1]{G_{{\lambda},{#1}}}
\newcommand \gmx[1]{G_{{\mu},{#1}}}
\newcommand \gm{\mathbb{G}_m}
\newcommand \gmS{\mathbb{G}_{m,S}}
\newcommand \gmK{\mathbb{G}_{m,K}}
\newcommand \gmSl{\mathbb{G}_{m,S_\lb}}
\newcommand \gmSlp{\mathbb{G}_{m,S_{\lb^p}}}
\newcommand \Ga{\mathbb{G}_a}
\newcommand \GaS{\mathbb{G}_{a,S}}

\newcommand \h[3]{H^{#1}(#2,#3)}
\newcommand \hxg[1]{H^{1}(X,#1)}

\newcommand \lb{\lambda}

\newcommand \muS{\mu_{p,S}}

\newcommand \N{\mathbb{N}}

\newcommand \Q{\mathbb{Q}}

\newcommand \Z{\mathbb{Z}}

\newcommand \ha{\mbox{$\hookrightarrow$}}
\newcommand \id{\operatorname{id}}

\newcommand \In{\subseteq}

\newcommand \too{\longrightarrow}

\newcommand \mTo {\longmapsto}

\newcommand \on{\stackrel}
\renewcommand \phi{\mbox{$\varphi$}}

\newcommand \pt   {\otimes}

\renewcommand \rho{\mbox{$\varrho$}}

   \def\clE{{\cal E}}
 \def\clG{{\cal G}}

\newcommand{\RW[2]}{W_{#1}{\Omega^{#2}_{X}}}

\newcommand \fr{\operatorname{F}}
\newcommand \rad{rad_{\mu,\lb}}

\newcommand \Gak{\mathbb{G}_{a,k}}
\newcommand  \apk {\alpha_{p,k}}

\def\scE{\mathscr{E}}

\DeclareMathOperator{\ext}{Ext}

 \DeclareMathOperator \im {Im}

\DeclareMathOperator{\Sp}{Spec}

\makeatletter
\renewcommand\l@subsection{\@tocline{2}{0pt}{5pc}{5pc}{}}
\makeatother

 


\begin{document}
\title[Models of $\mu_{p^2,K}$]{Models of $\mu_{p^2,K}$  over a discrete valuation ring}
\author[Dajano Tossici]{Dajano Tossici\\with an appendix by Xavier Caruso}
\begin{abstract}
Let $R$ be a discrete valuation ring  with residue field of
characteristic $p>0$. Let  $K$ be its fraction field. We prove
that any finite and flat $R$-group scheme,  isomorphic to
$\mu_{p^2,K}$ on the generic fiber, is the kernel in a short exact
sequence which generically coincides with the Kummer sequence. We
will explicitly describe and classify such models. In the appendix
X. Caruso shows how to classify models of $\mu_{p^2,K}$, in the
case of unequal characteristic, using the Breuil-Kisin theory.
\end{abstract}

\address{Dipartimento di Matematica, Università di Roma Tre, Rome, Italy}
\email{dajano.tossici@gmail.com} \maketitle
\begin{center}
\end{center}
\tableofcontents
\section*{Introduction}
 \textsc{Notation and Conventions.} 
If not otherwise specified we denote by $R$ a discrete valuation
ring (in the sequel d.v.r.) 
with residue field $k$ of
characteristic $p>0$. Moreover we write $S=\Sp(R)$. 
If for $n\in \N$ there exists a distinguished primitive $p^n$-th
root of unity $\zeta_n$ in a d.v.r. $R$,
 we call $\lambda_{(n)}:=\zeta_n-1$. 
 Moreover we suppose $\zeta_{i-1}=\zeta_i^p$ for $2\le i\le n$.
 We will denote by
$\pi\in R$ one fixed  uniformizer of $R$. 
All the schemes will be assumed noetherian. All the cohomological
groups are calculated in the fppf topology.
%

\vspace{1cm}

Let $K$ be a field of characteristic $0$ which contains a
primitive $p^n$-th root of unity:  this implies
$\mu_{p^n,K}\simeq(\Z/p^n\Z)_K$. We recall the following
 exact
sequence
$$
1\too \mu_{p^n,K}\too\gmK\on{p^n}{\too} \gmK\too 1,
$$
so-called the Kummer sequence. The Kummer theory says that this
sequence is universal, namely any $p^n$-cyclic Galois extension of
$K$ can be deduced by the Kummer sequence. We stress that this
sequence can be written also as follows
\begin{equation}\label{eq:Kummer sequence per n=2}
1\too \mu_{p^n,K}\too\gmK^n\on{\theta_n}{\too} \gmK^n\too 1
\end{equation}
where
$\theta_n((T_1,\dots,T_n))=(1-T_1^p,T_1-T_2^p,\dots,T_{n-1}-T_n^p)$.

Let $k$ be a field of characteristic $p>0$. The following exact
sequence
$$
0\too \Z/p^n\Z\too W_n(k)\on{\fr-1}\too W_n(k)\too 0,
$$
where $W_n(k)$ is the group scheme of Witt vectors of length $n$,
is called the Artin-Schreier-Witt sequence. The
Artin-Schreier-Witt theory implies that any $p^n$-cyclic Galois
covering of $k$ can be deduced by the Artin-Schreier-Witt
sequence.

 Let now $R$ be a d.v.r. of unequal characteristic which contains a $p^n$-th root of
unity. There exists a theory, the so called
Kummer-Artin-Schreier-Witt theory, which unifies the above two
theories. This means that there exists, for any $n$, a sequence
\begin{equation}\label{eq:KASWn}
0\too \Z/p^n\Z\too {\cal{W}}_{n}\too \cal{W}_{n}'\too 0,
\end{equation}
with ${\cal{W}}_{n}$ and $\cal{W}_n'$ smooth, which coincides with
the Kummer sequence over the generic fiber, and with the
Artin-Schreier-Witt sequence over the special fiber.

The case $n=1$ has been independently studied  by
Oort-Sekiguchi-Suwa (\cite{SOS}) and Waterhouse (\cite{Wat3}).
Later Green-Matignon (\cite{GM1}) and Sekiguchi-Suwa(\cite{SS4})
have, independently, constructed explicitly a unifying exact
sequence for $n=2$.   The case $n>2$  is treated in \cite{KASW2}
and \cite{KASW1}. The universality of these sequences is discussed
in \cite{SS2}. We remark that the explicit
Kummer-Artin-Schreier-Witt theory has been one of the main tools
used by Green-Matignon (\cite{GM1}) to give a positive answer to
the lifting problem for
cyclic Galois covers of order $p^2$. 

From now on $R$ is a d.v.r. with residue field of characteristic
$p>0$. In this paper we  study finite and flat $R$-group schemes
of order $p^2$ which are isomorphic to $\mu_{p^2,K}$ on the
generic fiber, i.e. models of $\mu_{p^2,K}$. And we prove that,
for any such a group scheme $G$, there exists an exact sequence
\begin{equation}\label{eq:sequenza esatta}
0\too G\too \clE_1 \too \clE_2\too 0,
\end{equation}
with $\clE_1,\clE_2$ smooth $R$-group schemes, which coincides
with the Kummer sequence on the generic fiber. So if $R$ is of
unequal characteristic and contains a primitive $p^2$-th root of
unity, the sequence \eqref{eq:KASWn}, with $n=2$, is the special
case when we consider the model $G=\Z/p^2\Z$. The sequence
\eqref{eq:sequenza esatta} is universal in the sense that any
$G$-torsor over a local scheme is obtained pulling back this
sequence (see \cite[\S 2.3]{io3}). Moreover we will describe
explicitly all such isogenies and their kernels and we will give a
classification of
models of $\mu_{p^2,K}$. 
The case of models of $\mu_{p,K}$ is well known.

 We now explain more precisely the classification we have
obtained.  In the first section we recall the definition of a
class
 of
group schemes of order $p$, called $G_{\lb,1}$ with $\lb \in
R\setminus\{0\}$, which are isomorphic to $\mu_{p,K}$ on the
generic fiber. If $R$ of unequal characteristic $\lb$ must also
satisfy the inequality $(p-1)v(\lb)\le v(p)$.   We remark that for
any of them there is an exact sequence
$$
0\too G_{\lb,1}\too \glb\too \glbp\too 0
$$
where $\glb$ and $\glbp$  are $R$-smooth models of $\gmK$.
Moreover on the generic fiber this sequence coincides with Kummer
sequence.
We  recall the following well known result.
\begin{propintro}$\ref{teo:modelli di Z/pZ}$.  If $G$ is a finite and
flat $R$-group scheme such that $G_K\simeq \mu_{p,K}$ then
$G\simeq G_{\lb,1}$ for some $\lb\in R\setminus\{0\}$.
\end{propintro}

One crucial remark is the following: any model of a diagonalizable
group scheme of order $p^2$ is an abelian extension of $G_{\mu,1}$
by $G_{\lb,1}$ for some $\mu,\lb\in R\setminus\{0\}$ (see
\ref{lem:modelli di Z/p^2 Z sono estensioni}). This leads us  to
study the group $\Ext^1_S{(G_{\mu,1},G_{\lb,1})}$.
 We stress we compute
  $\Ext^1_S(G,H)$  in the category of abelian fppf
 sheaves over $S$
 and this is not at all restrictive for our purposes by the above remark.

As one could imagine the above group of extensions is related to
the group $\Ext^1_S(\clG^{(\mu)},\clG^{(\lb)})$. In the second
section we essentially recall the theory of Sekiguchi-Suwa, which
gives an interesting interpretation of this group in terms of Witt
vectors. Moreover a few preliminary results are added.

In   \S \ref{sec:modelli di Z/p^2Z} we study the models of
$\mu_{p^2,K}$.
 First of all  we investigate on
$\Ext^1_S{(G_{\mu,1},G_{\lb,1})}$. We  distinguish four cases:
$$
\begin{array}{ll}
    \Ext^1_S(\mu_{p,S},\mu_{p,S}); \\
    \Ext^1_S(\mu_{p,S},\glx{1}), & \hbox{with  $\lb\in \pi R\setminus\{0\}$;} \\
    \Ext^1_S(\gmx{1},\mu_{p,S}), & \hbox{with $\mu \in \pi R \setminus\{0\}$;} \\
    \Ext^1_S(\gmx{1},\glx{1}), & \hbox{with $\mu,\lb \in \pi R \setminus\{0\}$.} \\
\end{array}%
$$
The first three cases are easy and already known. The fourth case
is  more subtle. If $\lb'\in \pi R$, let $S_{\lb'}:=\Sp(R/\lb' R)$
and let $i_{\lb'}$ be the closed immersion $i_{\lb'}:S_{\lb'} \too
S$. Let us suppose that $\mu,\lb \in \pi R\setminus \{0\}$.  We
define the group
\[\rad:=\left\{
\begin{array}{c}
\text{$(F(T),j)\in
\Hom_{S_{\lambda}-{{gr}}}(i_{\lb}^*G_{\mu,1},\mathbb{G}_{m,S_{\lambda}})\times \Z/p\Z$ such that}\\[1mm]
F(T)^p(1+\mu T)^{-j}=1\in
\Hom_{S_{\lambda^p}-{{gr}}}(i_{\lb^p}^*G_{\mu,1},\
G_{m,S_{\lambda^p}})
\end{array}\right\}\Bigl/\langle 1+\mu T \rangle\]
In the rest of the paper, with a slight abuse of notation, for  a
group scheme $G$ over $R$ and $\lb'\in \pi R$, we will sometimes
simply write $G$ to indicate its restriction over $S_{\lb'}$,
instead of writing $i_{\lb'}^* G$. However it will be clear from
the context what we mean and it should not cause any confusion.

Let  $(F(T),j)\in \rad$ and $\widetilde{F}(T)\in R[T]$ any its
lifting. We will explicitly define in \S\ref{subsec:explicit
description...}  a finite and flat $R$-group scheme
$\clE^{(\mu,\lb;\widetilde{F},j)}$ which is extension of
$G_{\mu,1}$ by $G_{\lb,1}$. Its extension class
$[\clE^{(\mu,\lb;\widetilde{F},j)}]$ in
$\ext^1_S(G_{\mu,1},G_{\lb,1})$ does not depend on the choice of
the lifting. So we will sometimes denote it  simply with
$[\clE^{(\mu,\lb;{F},j)}]$, omitting to specify the lifting. The
above group scheme is the kernel of an isogeny of smooth group
schemes of dimension $2$. This isogeny is generically isomorphic
to the Kummer sequence, in the form of \eqref{eq:Kummer sequence
per n=2}.
 Using this notation, we  give a description of $\Ext^1_S(G_{\mu,1},\glx{1})$.
\begin{thmintro}$\ref{teo:ext1(glx,gmx)}$. Let  $\mu,\lb\in \pi R\setminus\{0\}$. If $char(R)=0$, we suppose \mbox{$(p-1)v(\lb),(p-1)v(\mu)\le
v(p)$.}
Then  there exists an exact sequence
\begin{equation*}
\begin{array}{ll}
0\too \rad \on{\beta}{\too}&
\Ext^1_S(G_{\mu,1},\glx{1})\too\\
&\too \ker \bigg(H^1(S,G_{\mu,1}^\vee)\too
H^1(S_\lb,G_{\mu,1}^\vee)\bigg) 
\end{array}
\end{equation*}
 where $\beta$ is defined by
$$
(F,j)\longmapsto [\clE^{(\mu,\lb;F,j)}].
$$
In particular the set $\{[\clE^{(\mu,\lb;F,j)}]; (F,j)\in
\rad\}\In \Ext^1_S(G_{\mu,1},\glx{1})$ is a group  isomorphic to
$\rad$.
\end{thmintro}

The group $\Ext^1_S(G_{\mu,1},\glx{1})$   has been described by
Greither in \cite{gre} through  a short exact sequence, different
from the one of the previous theorem. An advantage of our
description is that  we individuate a big  class of extensions.
For instance in unequal characteristic this class "essentially"
covers all the group schemes of order $p^2$. Indeed from
\ref{teo:ext1(glx,gmx)} it follows that in such a case
 any group scheme of
order $p^2$, up to an extension of  $R$, is of the form
$\clE^{(\mu,\lb;\widetilde{F},j)}$ (see \ref{rem:essenzialmente
tutti i gruppi di ordine p2}).

The following result is obtained combining the previous Theorem
and the Sekiguchi-Suwa theory.
\begin{propintro}$\ref{prop:rad_p}$
Let $\mu ,\lb \in \piR$. If $char(R)=0$ we suppose
$(p-1)v(\mu),(p-1)v(\lb)\le v(p)$. Then the group
$\{[\clE^{(\mu,\lb;F,j)}]; (F,j)\in
\rad\} $ is isomorphic to 
\begin{align*}
\Phi_{\mu,\lb}:=\bigg\{(a,j)\in (R/\lb
R)^{{\fr}-\mu^{p-1}}\times\Z/p\Z& \text{ such that }
pa-j\mu=\frac{p}{\mu^{p-1}}a^p\in R/\lb^p
R\bigg\}\bigg/\left<(\mu,0)\right>,
\end{align*}
through the morphism
$$
(a,j)\longmapsto [\clE^{(\mu,\lb;F(T),j)}],
$$
where
$F(T)=1+\sum_{i=1}^{p-1}\frac{\prod_{k=0}^{i-1}(a-k\mu)}{i!}T^i$.
\end{propintro}
%

It could be possible that two different elements of
$\Ext^1_S(G_{\mu,1},G_{\lb,1})$ are isomorphic just as group
schemes.  We study when this happens, in the case of models of
$\mu_{p^2,K}$, and we prove the following theorem.
\begin{thmintro}$\ref{cor:modelli di Z/p^2 Z sono cosi'2}$
 Let $G$ be a finite and flat $R$-group
scheme such that $G_K\simeq \mu_{p^2,K}$. Then $G\simeq
\clE^{(\pi^{m},\pi^{n};\tilde{F},1)}$ for some $m,n\ge 0$,
$\tilde{F}(T)$ a lifting of
$F(T)=\sum_{k=0}^{p-1}\frac{a^k}{k!}T^k$ with $(a,1)\in
\Phi_{\pi^m,\pi^n}$. If $char(R)=0$ then $m\ge n$ and $(p-1)m\le
v(p)$, while if $char(R)=p$ then $m\ge pn$. Moreover $m,n$ and
$a\in R/\pi^{n} R$ are unique.
\end{thmintro}
The last section is devoted to determine, through the description
of \ref{prop:rad_p}, the special fibers of the extensions
$[\clE^{(\mu,\lb;F,j)}]$.

We want to remark that sometimes we treat the cases of equal and
unequal characteristic separately. But it could be possible to
treat them always together. This relies on the fact that results
of Sekiguchi and Suwa work over $\Z_{(p)}$-algebras. However, for
reader convenience, we prefer  sometimes to  separate the two
cases;  this could also  allow the reader to appreciate how the
equal characteristic case is in fact easier.



We finally stress that  Breuil and Kisin (see \cite{breuil},
\cite{kisin}, \cite{kisin2})  have recently developed
 a theory which classifies finite and flat commutative  group schemes over a complete
 d.v.r. of unequal characteristic with perfect residue field. This classification is obtained through
 some objects of linear algebra which are very easy to study.
In the appendix of this paper  X.Caruso shows, using this theory,
how to obtain in a very efficient and elegant way a result similar
to
 \ref{cor:modelli di Z/p^2 Z sono cosi'2} (under the further assumptions $R$ complete, of unequal characteristic and $k$
 perfect).  
 Unfortunately with this purely algebraic approach one looses some important geometric aspects, namely the explicit description
 of these group
 schemes, the geometric meaning of parameters which appear in the classification  and
 the existence of a universal sequence from which one can determine any torsor under
 these group schemes.

In principle our approach could be applied, for instance, also for
the classification of models of $\mu_{p^n,K}$ with $n>2$. But the
technical problems which can arise could be very difficult. Over a
complete d.v.r.of unequal characteristic with perfect residue
field maybe the right way is just to mix the above two theories:
one can obtain the parameters which classify these models using
the very efficient theory of Breuil-Kisin and then one can try to
understand geometrically these parameters through the
Sekiguchi-Suwa theory, in  such a way to obtain explicit
descriptions of these models.

As recalled at the beginning of the introduction, if $char(R)=0$
and $R$ contains a primitive root of unity, then
$\mu_{p^2,K}\simeq
(\Z/p^2\Z)_K$. In such a case the explicit description 
of the models of $(\Z/p^2\Z)_K$ presented in this paper has been
used in \cite{io3} to study the degeneration of
$(\Z/p^2\Z)_K$-torsors from characteristic $0$ to characteristic
$p$  and in particular the problem of extension of
$(\Z/p^2\Z)_K$-torsors.

\textbf{Acknowledgments} This paper constitutes part of my PhD
thesis. It is a pleasure to thank Professor Carlo Gasbarri for his
guidance and his constant encouragement. I  begun this work during
my visit to Professor Michel Matignon in 2005 at the Department of
Mathematics of the University of Bordeaux 1. I am indebted to him
for the useful conversations  and for his great interest in my
work.  I am deeply grateful  to Matthieu Romagny for  his very
careful reading of this paper and for his several comments,
suggestions, remarks and answers to my questions which helped me
to improve the paper. I wish also to thank Professor Fabrizio
Andreatta for having suggested me the sketch of proof of
\ref{prop:torsori schemi normali}.   I thank Filippo Viviani and
Marco Antei for the stimulating and useful conversations, and
Xavier Caruso for
  his appendix to this paper. Finally I acknowledge the referee
   for providing constructive comments and help in improving the exposition.

\section{Models of $\mu_{p,K}$}

\subsection{Some group schemes of order $p$}\label{sec:glbn}

For any $\lambda\in R$ we define the group scheme
$$
\glb:=\Sp(R[T,\frac{1}{1+\lambda T}])
$$
The  $R$-group scheme  structure is given by
\begin{align*}
&T\too 1\pt T+T\pt 1 +\lb T\pt T\qquad&\text{comultiplication}\\
&T\too 0 \qquad& \text{counit}\\
&T\too -\frac{T}{1+\lb T}\qquad& \text{coinverse}
\end{align*}

We observe that if $\lb=0$ then $\glb\simeq \GaS$. It is possible
to prove that $\glb\simeq \gmu$ if and only if $v(\lambda)=v(\mu)$
and the isomorphism is given by $T\too \frac{\lambda}{\mu}T$.
Moreover it is easy to see that, if $\lb\in \pi R\setminus\{ 0\}$,
then $\glb_k\simeq \mathbb{G}_{a,k}$ and $\glb_K\simeq
\mathbb{G}_{m,K}$. When $\lb\in R\setminus\{0\}$ we can define the
morphism
\begin{align*}
\alpha^{(\lb)}:\glb\too \gmS
\end{align*}
given, on the level of Hopf algebras,  by  $T\mTo 1+\lb T$: it is
an isomorphism on the generic fiber. If $v(\lb)=0$ then
$\alpha^{(\lb)}$ is an isomorphism.

We now define some finite and flat group schemes of order $p$. If
$R$ is of unequal characteristic we will assume that $\lambda\in
R$ satisfies the condition
 $$
\qquad v(p)\ge (p-1)v(\lambda).
$$
Let $\lb \in R\setminus \{0\}$. Then the morphism
\begin{align*}
\psi_{\lb,1}:&\glb\too \g^{(\lb^{p})}\\
             &T\too P_{\lb,1}(T):=\frac{(1+\lambda T)^{p}-1}{\lambda^{p}}
\end{align*}
is an isogeny of  degree $p$. Let
$$
\glx{1}:=\Sp(R[T]/P_{\lb,1}(T))
$$
be  its kernel. It is a commutative finite flat group scheme over
$R$ of order $p$. It is easy to generalize this definition to
obtain some group schemes of order $p^n$, the so called
$G_{\lb,n}$.   We observe that $\alpha^{(\lambda)}$ is compatible
with $\psi_{\lb,1}$, i.e
 the following  diagram
\begin{equation*}
\xymatrix@1{\glb\ar[d]_{\psi_{\lb,1}}  \ar[r]^{\alpha^{(\lb)}}&\ar[d]^{p}\gmS\\
\g^{(\lb^{p})} \ar[r]^(.6){\alpha^{(\lb^{p})}}&\gmS}
\end{equation*}
is commutative.
 Then it   induces a  morphism
$$
\alpha^{(\lb)}:\glx{1}\too\mu_{p,S}
$$
which is an isomorphism on the generic fiber. And if $v(\lb)=0$
 it  is an isomorphism. Moreover it is possible to prove that
\begin{align*} 
{(G_{\lb,1})}_k\simeq \alpha_{p,k} &\qquad\text{ if $char(R)=0$ and  $v(p)>(p-1)v(\lb)>0$, or $char(R)=p$ and $v(\lb)>0$};\\
G_{\lb,1} \quad \text{étale} &\qquad \text{ if $char(R)=0$ and
$(p-1)v(\lb)=v(p)$}.
\end{align*}
We also recall that
\begin{equation}\label{eq:Hom(glbn,glmn)}
 \Hom_{S-gr}(G_{\lb,1},G_{\lb',1})\simeq \Hom_{S-gr}(G_{\lb,1},{\clG^{(\lb')}})\simeq %
\left\{\begin{array}{ll}%
        \Z/p\Z,& \text{ if } v(\lb)\ge v(\lb') \\
        0,  &  \text{ if } v(\lb)< v(\lb').\\
\end{array}%
\right.
\end{equation}
If $v(\lb)\ge v(\lb')$ the morphisms are given by
\begin{align*}
\sigma_j: G_{\lb,1}&\too G_{\lb',1}\\
 T&\longmapsto \frac{(1+\lb T)^{j}-1}{\lb'}.
\end{align*}
 It  follows easily that
$G_{\lb,1}\simeq G_{\lb',1}$ if and only if $v(\lb)=v(\lb')$.

 In the following, if $char(R)=0$,  any time we will speak
about $G_{\lb,1}$  it will be assumed that $(p-1)v(\lb)\le v(p)$.
If  $R$ contains a distinguished primitive $p$-{th} root of unity
$\zeta_1$  then, since \mbox{$v(p)=(p-1)v(\lb_{(1)}),$} the above
condition  is equivalent to $v(\lb)\le v(\lb_{(1)})$. We recall
that by definition $\lb_{(1)}:=\zeta_1-1$.
\subsection{Classification of models of ${\mu_{p,K}}$.}

\begin{defn}
Let $H$ be a group scheme over $K$. Any  flat $R$-group scheme $G$
such that $G_K\simeq H$ is called a \textsl{model} of $H$. If $H$
is finite over $K$ we also require $G$ finite over $R$.
\end{defn}
\begin{rem}It has been proved by Waterhouse and Weisfeiler, in
\cite[2.5]{Wat}, that
 any smooth model of $\gmK$ with connected fibers is isomorphic  to $\glb$ for some $\lb\in R\setminus\{0\}$.
\end{rem}

It follows from \cite[2.4]{Wat} that 
$\mu_{p,S}$ is the minimal model for $\mu_{p,K}$, i.e. for any
model $G$ of ${\mu_{p,K}}$ we have a model map, namely a morphism
which is
an isomorphism on the generic fiber, 
$$
G\too \mu_{p,S}.
$$
We now recall the classification of $\mu_{p,K}$-models, which is
due to Waterhouse-Weisfeiler (\cite[\S 2]{Wat}).
\begin{prop}\label{teo:modelli di Z/pZ}
If $G$ is a finite and flat  $R$-group scheme such that $G_K\simeq
\mu_{p,K}$ then $G\simeq G_{\lb,1}$ for some $\lb\in
R\setminus\{0\}$.
\end{prop}
\begin{proof}
As  remarked  above we have an $R$-model map $ \phi:G\too
\mu_{p,S}. $ By \cite[1.4]{Wat} it is a composition of a finite
number of Néron blow-ups. Now, let us consider the group scheme
$G_{\mu,1}$. The trivial subgroup scheme $H=e$ is the unique
subgroup $H$ of $(G_{\mu,1})_k$ which gives a nontrivial blow-up.
If $G_{\mu,1}$ is not étale, which means $char(R)=p$ or
$char(R)=0$ and $(p-1)v(\mu)<v(p)$, it is easy to see that the
Néron blow up of $G_{\mu,1}$ at $e$ is $ G_{\mu\pi,1}$. In the
étale case we obtain a not finite group scheme.  So $G\simeq
G_{\lb,1}$ with $v(\lb)$ equal to the number of (nontrivial)
blow-ups from $\mu_{p,S}$.
\end{proof}


We observe that if $char(R)=0$ (resp. $char(R)=p$) there are a
finite (resp. infinite) number  of models of $\mu_{p,K}$.

\section{Preliminaries}\label{par:review of results}
The main point to classify models of $\mu_{p^2,K}$ will be the
study of extensions of $G_{\mu,1}$ by $G_{\lb,1}$. As we will see
in \ref{lem:modelli di Z/p^2 Z sono estensioni} any such an
extension is abelian. Therefore, for simplicity,  we restrict to
study abelian extensions. Hence in the following, for any  scheme
$X$ and $G,H$ flat commutative group schemes over $X$, we define
$\Ext_X^*(G,H)$ as the right derived functor of $H\longmapsto
\Hom_{X-gr}(G,H)$ on the abelian fppf-sheaves over $X$ (see
\cite{Br}). We remark that if $H$ is affine over $X$ then any
element of $\Ext_X^1(G,H)$ is representable by a flat commutative
group scheme over $X$ (see \cite[III 17.4]{Oo}).

In this section we recall some results about some groups of
extensions  which will play a key role in the next section.

\subsection{Sekiguchi-Suwa Theory}\label{sec:Sekiguchi-Suwa theory}


We here briefly recall some results about $\Ext^1_S(\gmu,\glb)$
with $\mu,\lb \in R\setminus\{0\}$.   We now distinguish four
cases:
$$
\begin{array}{ll}
     \Ext^1_S(\gmS,\gmS)   ; \\
     \Ext^1_S(\gmu,\gmS)  , &   \text{ with } \mu\in \pi R\setminus\{0\}; \\
     \Ext^1_S(\gmS,\glb)  , &   \text{ with } \lb\in \pi R\setminus\{0\}; \\
     \Ext^1_S(\gmu,\glb) , &   \text{ with } \mu,\lb\in \pi
     R\setminus\{0\}.
\end{array}%
$$
The first three groups are trivial, see for instance \cite[I 2.7,
II 1.4]{SS6}. The  fourth case is more subtle and it has been
studied in
 \cite{SS9} and
\cite{SS4}. We now suppose $\mu,\lb\in \pi R\setminus\{0\}$. For
any $\lb'\in \pi R\setminus\{0\}$  set $S_{\lb'}=\Sp(R/\lb' R)$.
 Let us now consider the exact sequence on the fppf site  over $S$
\begin{equation}\label{eq:succ esatta glb in gm}
0\too \glb\on{\alpha^{(\lb)}}{\too}{\gmS}\too
i_{\lb}^*{{\mathbb{G}}_{m,S_{\lb}}}\too 0,
\end{equation}
 where $i_{\lb}$ denotes the closed immersion $S_\lb:=\Sp(R/\lb R)\ha S$
(see \cite[II 1.1]{SOS}). We observe that by definitions we have
that
$$
\Hom_{S_{\lb}-gr}(\gmu,{\gmSl})=\{F(T)\in (R/\lb
R[T])^*|F(U)F(V)=F(U+V+\mu UV) \}.
$$
Here there is the abuse of notation we mentioned in the
introduction. Indeed if it is clear from the context, for a group
scheme $G$ over $R$  we sometimes simply write $G$ to indicate its
restriction over $S_\lb$, instead of writing $i_{\lb}^* G$.

 If
we apply the functor $\Hom_{S-gr}(\gmu,\cdot)$ to the sequence
\eqref{eq:succ esatta glb in gm} we obtain, in particular,  a
 morphism  
\begin{equation*}
\begin{aligned}
\Hom_{S_\lb-gr}(\gmu,{\gmSl})
\on{\alpha}{\too}\Ext^1_S(\gmu&,\glb),
\end{aligned}
\end{equation*}
 given by
$$
F\longmapsto [\clE^{(\mu,\lb;\widetilde{F})}],
$$
where $\widetilde{F}(T)\in R[T]$ is a lifting of $F(T)$ and $
\clE^{(\mu,\lb;\widetilde{F})} $ is an $R$-smooth affine
commutative group defined as follows:
$$
\clE^{(\mu,\lb;\widetilde{F})}:= \Sp(R[T_1,T_2,\frac{1}{1+\mu
T_1},\frac{1}{\widetilde{F}(T_1)+\lb T_2}])
$$
\begin{enumerate}
    \item  comultiplication
    \begin{align*}
    T_1\longmapsto &T_1\pt 1+1\pt T_1+\mu T_1\pt T_1\\
    T_2\longmapsto &T_2\pt \widetilde{F}(T_1)+\widetilde{F}(T_1)\pt T_2 +\lb T_2\pt T_2+\\
                   & \qquad   \quad    \frac{\widetilde{F}(T_1)\pt \widetilde{F}(T_1)-\widetilde{F}(T_1\pt 1+1\pt T_1+\mu T_1\pt T_1)}{\lb}
    \end{align*}
    \item counit
    \begin{align*}
    &T_1\longmapsto 0\\
    &T_2\longmapsto \frac{1-\widetilde{F}(0)}{\lb}
    \end{align*}
    \item coinverse
    \begin{align*}
    &T_1\longmapsto -\frac{T_1}{1+\mu T_1}\\
& T_2 \longmapsto\frac{\frac{1}{\widetilde{F}(T_1)+\lb
T_2}-\widetilde{F}(-\frac{T_1}{1+\mu T_1})}{\lb}
\end{align*}
\end{enumerate}
We moreover  define the following  homomorphisms of group schemes
$$
\glb=\Sp(R[T,(1+\lb T)^{-1}])\too \clE^{(\mu,\lb;\widetilde{F})}
$$
by
\begin{align*}
&T_1\longmapsto 0\\
&T_2\longmapsto T +\frac{1-\widetilde{F}(0)}{\lb}
\end{align*}
and
$$
\clE^{(\mu,\lb;\widetilde{F})}\too \gmu=\Sp(R[T,\frac{1}{1+\mu
T}])
$$
 by
\begin{align*}
&T\too T_1.
\end{align*}
It is easy to see that
\begin{equation*}
0\too \glb\too \clE^{(\mu,\lb;\widetilde{F})}\too \gmu\too 0
\end{equation*}
is exact. A different choice of the lifting $\widetilde{F}(T)$
gives an isomorphic extension.  We recall the following theorem.
\begin{thm}\label{teo:ss ext1}For any $\mu,\lb\in \pi R\setminus\{0\}$, the map $$\alpha:\Hom_{S_\lb-gr}(\gmu,{\gmSl}){\too}
\Ext^1_S(\gmu,\glb)$$  is a surjective morphism of groups. And
$\ker(\alpha)$ is generated by the class of $1+\mu T$.
\end{thm}
\begin{proof}
\cite[\S 3]{SS9}.
\end{proof}
\begin{rem}
The group scheme $\clE^{(\mu,\lb;1)}$ is well defined also in the
case
 $v(\mu)=0$ or $v(\lb)=0$ and it is isomorphic to $\gmu\times \glb$.
\end{rem}
 \noindent We now define some spaces  used by Sekiguchi and Suwa to describe
 $\Hom_{S_\lb-gr}(\gmu,{\gmSl})$. 
 See \cite{SS4} for details.
\begin{defn}
For any ring $A$, let $W(A)$ be  the ring of infinite Witt
vectors. We define
$$
\widehat{W}(A)=\left\{(a_0,\dots,a_n,\dots)\in W(A)\;;
\begin{array}{c} \text{$\ a_i$ is
nilpotent for any $i$  and } \\[1mm]
 \text{$a_i=0$ for all but a finite number of $i$}
\end{array}\right\}.
$$
\end{defn}

We recall the definition of the so-called Witt-polynomial: for any
$r\ge 0$ it is
$$
\Phi_r(T_0,\dots,T_r)=T_0^{p^r}+pT_{1}^{p^{r-1}}+\dots+p^rT_r.
$$
Then  the following maps are defined:
\begin{itemize}
    \item[-]Verschiebung
    \begin{align*}
    V:W(A)&\too W(A)\\
(a_0,\dots,a_n,\dots)&\longmapsto (0,a_0,\dots,a_n,\dots)
    \end{align*}
    \item[-]Frobenius
 \begin{align*}
    \fr:W(A)&\too W(A)\\
(a_0,\dots,a_n,\dots)&\longmapsto
(F_0(\mathbf{T}),F_1(\mathbf{T}),\dots,F_n(\mathbf{T}),\dots),
    \end{align*}
where the polynomials $F_r(\mathbf{T})=F_r(T_0,\dots,T_r)\in
\Q[T_0,\dots,T_{r+1}]$ are defined inductively by
$$
\Phi_r(F_0(\mathbf{T}),F_1(\mathbf{T}),\dots,F_r(\mathbf{T}))=\Phi_{r+1}(T_0,\dots,T_{r+1}).
$$
\end{itemize}
If $p=0\in A$ then $\fr$ is the usual Frobenius. The ideal
$\widehat{W}(A)$  is stable with respect to these maps.

For any morphism $G: \widehat{W}(A)\too \widehat{W}(A)$ we will
set $\widehat{W}(A)^G:=\ker G$. And for any $a\in A$ we denote
 the element $(a,0,0,\dots,0,\dots)\in W(A)$ by $[a]$. It is
 called the Teichm\"{u}ller representant of $a$.
The following lemma  will be useful later.
\begin{lem}\label{lem:somma termine per termine} 
Let $\lb \in \pi R\setminus\{0\}$. If $char(R)=0$ we suppose $
(p-1)v(\lb)\le pv(p)$. For any
$\bbf{a}=(a_0,a_1,\dots),\bbf{b}=(b_0,b_1,\dots)\in
\widehat{W}(R/\lb R)^{\fr}$ then
$$
\bbf{a}+\bbf{b}=(a_0+b_0,a_1+b_1,\dots, a_i+b_i,\dots)
$$
\end{lem}
\begin{proof}
We suppose that $\bbf{a}+\bbf{b}=(c_0,c_1,\dots,c_i,\dots)$. If
$U_i$ and $V_i$ have weight $p^i$ then
 we have that $c_r(\mathbf{U},\mathbf{V})$ is isobaric
of weight $p^r$. By definition of sum between Witt vectors we have
$$c_r(\bbf{a},\bbf{b})=a_r+b_r+
c'_r((a_0,a_1,\dots,a_{r-1}),(b_0,b_1,\dots,b_{r-1}))$$ for some
polynomial $c'_r(U_0,\dots, U_{r-1},V_0,\dots,V_{r-1})$ isobaric
of weight $p^r$. Hence $\deg (c'_r)\ge p$.

Let
$\widetilde{\bbf{a}}=(\widetilde{a}_0,\widetilde{a}_1,\dots)$,$\widetilde{\bbf{b}}=(\widetilde{b}_0,
\widetilde{b}_1,\dots)\in W(R)$ be liftings of $\bbf{a}$ and
$\bbf{b}$, respectively. For any $r\ge 1$, up to changing
$\widetilde{\bbf{a}}$ with $\widetilde{\bbf{b}}$, we can suppose
that
$v(\widetilde{a}_k)=\min\{v(\widetilde{a}_i),v(\widetilde{b}_i)|
i=0,\dots,r-1\}$, for some $0\le k \le r-1$. Since $\deg c'_r\ge
p$ then $v( c'_r(\widetilde{\bbf{a}},\widetilde{\bbf{{b}}}))\ge
pv(\widetilde{a}_k)$. We claim that, since $\fr(\bbf{a})=0$, then
$pv(\widetilde{a}_k)\ge v(\lb)$. If $char(R)=p$ this is clear. We
now suppose $char(R)=0$. From \cite[1.2.1]{SS4} we have that
$pv(\widetilde{a}_k)\ge \min \{v(p)+v(\widetilde{a}_k),v(\lb)\}$.
This implies, by hypothesis on $v(\lb)$, $pv(\widetilde{a}_k)\ge
v(\lb)$. Hence $c'_r(\bbf{a},\bbf{b})=0\in R/\lb R$.  So
$$
\bbf{a}+\bbf{b}=(a_0+b_0,a_1+b_1,\dots, a_i+b_i,\dots)
$$
\end{proof}
  We now recall the definition of the
Artin-Hasse exponential series
$$
E_p(T):=\exp\bigg(\sum_{r\ge
0}\frac{T^{p^r}}{p^r}\bigg)=\prod_{r=0}^{\infty}\exp\bigg(\frac{T^{p^r}}{p^r}\bigg)\in
\Z_{(p)}[[T]].
$$
Sekiguchi and Suwa introduced  a deformation of the Artin-Hasse
exponential map in \cite{SS4}. 
They defined $E_p(U,\Lambda;T)\in \Q[U,\Lambda][[T]]$ by
$$
E_p(U,\Lambda;T):=(1+\Lambda
T)^{\frac{U}{\Lambda}}\prod_{r=1}^{\infty}(1+\Lambda^{p^r}T^{p^r})^{\frac{1}{p^r}((\frac{U}{\Lambda})^{p^r}-(\frac{U}{\Lambda})^{p^{r-1}})}
$$
They proved that $E_p(U,\Lambda;T)$ has in fact its coefficients
in $\Z_{(p)}[U,\Lambda]$. It is possible to show (\cite[2.4]{SS4})
that
\begin{equation}\label{eq:E_p(a,mu;T)}E_p(U,\Lambda;T)=\left\{%
\begin{array}{ll}
    \prod_{(i,p)=1}E_p(U\Lambda^{i-1}T^i)^\frac{(-1)^{i-1}}{i}, & \hbox{if $p>2$;} \\
    \prod_{(i,2)=1}E_p(U\Lambda^{i-1}T^i)^\frac{1}{i}\bigg[\prod_{(i,2)=1}E_p(U\Lambda^{2i-1}T^{2i})^\frac{1}{i}\bigg]^{-1}, & \hbox{if $p=2$.} \\
\end{array}%
\right.
\end{equation}
Let $A$ be a $\Z_{(p)}$-algebra and $a,\mu\in A$. We define
$E_p(a,\mu;T)$ as $E_p(U,\Lambda;T)$ evaluated at $U=a$ and
$\Lambda=\mu$.

\vspace{.5cm}

If \textbf{\textit{a}}$=(a_0,a_1,a_2,\dots)\in W(A)$ we define the
formal power series
\begin{equation*}
E_p(\textbf{\textit{a}},\mu;T)=\prod_{k=0}^\infty
E_p(a_k,\mu^{p^k};T^{p^k}).
\end{equation*}
We stress that if $\mu$ is nilpotent in $A$ then $\bbf{a}\in
\widehat{W}(A)$ if and only if  $E_p(\textbf{\textit{a}},\mu;T)$
is a polynomial (see \cite[2.11]{SS4}). The following  result
gives an explicit description of
 $\Hom_{S_{\lb}-gr}(\gmu,{\mathbb{G}_{m,S_{\lb}}})$.

\begin{thm}\label{teo:ss hom} Let $\mu,\lb \in \pi R\setminus\{0\}$. 
The homomorphism
\begin{align*}
\xi^0_{R/\lb R}:\widehat{W}(R/\lb R)^{\fr-[\mu^{p-1}]}&\too \Hom_{S_{\lb}-gr}(\gmu,{\mathbb{G}_{m,S_{\lb}}})\\
\textbf{a}&\longmapsto E_p(\textbf{a},\mu;T)
\end{align*}
is bijective.
\end{thm}
\begin{proof}
It is a particular case of \cite[2.19.1]{SS4}.
\end{proof}

Moreover \ref{teo:ss ext1} and \ref{teo:ss hom} give the
following:
\begin{cor} For any $\mu,\lb \in \pi R\setminus\{0\}$ the map\begin{align*}
\alpha\circ \xi_{R/\lb R}^0 :\widehat{W}(R/\lb R)^{\fr-[\mu^{p-1}]}/\left<1+\mu T\right>&\too \ext^1_S(\gmu,{\glb})\\
\textbf{a}&\longmapsto [\clE^{(\mu,\lb;\widetilde{F})}],
\end{align*}
where $\widetilde{F}(T)\in R[T]$ is a lifting of
$E_{p}(\bbf{a},\mu; T)$, is an isomorphism.
\end{cor}

We now describe some natural maps through these identifications.
Let $\mu,\lb \in R\setminus\{0\}$ and, if \mbox{$char(R)=0$}, let
us assume that $(p-1)\mu\le v(p)$. Consider the isogeny
$$
\psi_{\mu,1}:\gmu\too \gmup.
$$
Let us first suppose that $p>2$. Then we have, by \cite[1.4.1 and
3.8]{SS4}), that,  if $p^2\equiv 0 \mod \lb$,
$$
\psi_{\mu,1}^*:\Hom_{{S_{\lb}}-gr}(\gmup,{\gmSl})\too
\Hom_{S_\lb-gr}(\gmu,{\gmSl})
$$
is given by
\begin{equation}\label{eq:phi*:W(R)too W(R)}
\begin{aligned}
\textit{\textbf{a}}&\longmapsto[\frac{p}{\mu^{p-1}}]\textit{\textbf{a}}+V(\textit{\textbf{a}}),
&\text{ if $char(R)=0$ and} \ \\
%
\textit{\textbf{a}}&\longmapsto V(\textit{\textbf{a}}), &\text{ if
$char(R)=p$.}\qquad
\end{aligned}
\end{equation}

For $p=2$ the situation is slightly different. Let us define
 a variant of the Verschiebung as follows. Define polynomials
 $$
\widetilde{V}_r(\mathbf{T})=\widetilde{V}_r(T_0,\dots,T_r)\in
\Q[T_0,\dots,T_{r}]
$$
 inductively  by $\widetilde{V}_0=0$ and
$$
\Phi_r(\widetilde{V}_0(\mathbf{T}),\dots,\widetilde{V}_r(\mathbf{T}
))=2^{2^r}\Phi_{r-1}(T_0,\dots,T_{r-1})
$$
for $r\ge 1$. It has been proved in \cite[1.4.1]{SS4} that, if
$\bbf{a}=(a_0,a_1,\dots)$, then
\begin{equation}\label{rem: tilde V per p=2}\widetilde{V}(\textit{\textbf{a}})\equiv
(0,2a_0,2a_0^2,\dots,2a_0^{{2^{i-1}}},\dots)\mod 2^2
\end{equation}
In particular $\widetilde{V}(\textit{\textbf{a}})\equiv 0\mod 2$.
Then we have, by  \cite[3.8]{SS4}, that (with possibly
$2^2\not\equiv 0\mod \lb$)
$$
\psi_{\mu,1}^*:\Hom_{S_{\lb}-gr}(\clG^{(\mu^2)},{\gmSl})\too
\Hom_{S_{\lb}-gr}(\gmu,{\gmSl})
$$
is given by
\begin{equation}\label{eq:psi per p=2}
\begin{aligned}
\textit{\textbf{a}}&\longmapsto[\frac{2}{\mu}]\textit{\textbf{a}}+V(\textit{\textbf{a}})+\widetilde{V}(\textit{\textbf{a}}),
&\text{ if $char(R)=0$ and}   \   \\
\textit{\textbf{a}}&\longmapsto V(\textit{\textbf{a}}), &\text{if
$char(R)=p$.}\qquad
\end{aligned}
\end{equation}


We now come back to the general case $p\ge 2$. 
Consider the morphism
\begin{equation*}\label{eq:def p}
\begin{aligned}
 p:\Hom_{S_\lb-gr}(\gmu,{\gmSl})&\too\Hom_{S_{\lb^p}-gr}(\gmu,{\mathbb{G}_{m,S_{\lb^p}}})\\
     F(T)&\longmapsto F(T)^p.
     \end{aligned}
\end{equation*}
We observe that we have
$$
{\psi_{\lb,1}}_*\circ \alpha=\alpha\circ p.
$$
Let \textbf{\textit{a}}$\in (\widehat{W}(R/\lb
R))^{\fr-[\mu^{p-1}]}$ and take any lifting
$\widetilde{\textbf{\textit{a}}}\in W(R)$. Using the
identifications of \ref{teo:ss hom} the morphism $p$ above
is given by
\begin{equation}\label{eq:p a}
\textbf{\textit{a}}\longmapsto p\widetilde{\textbf{\textit{a}}}
\end{equation}
(see \cite[4.6]{SS4}).  We will sometimes simply write $p\bbf{a}$.
\subsection{Interpretation of
$\Ext^1_S(G_{\mu,1},\gmS)$}

First of all, we   briefly recall a useful spectral sequence. Let
$G,H$ be flat and commutative $S$-group schemes and let
$\mathcal{E}xt_S^i(G,H)$ denote the fppf-sheaf on $Sch_{/S}$,
associated to the presheaf $X\longmapsto
\Ext^i_S{(G\times_{S}X,H{\times_S}X)}$. Then we have a spectral
sequence
$$
E_{2}^{ij}=H^i(S,\mathcal{E}xt_S^j(G,H))\Rightarrow
\Ext^{i+j}_S(G,H),
$$
which  in low degrees gives
\begin{equation}\label{eq:spectral sequence per ext}
\begin{aligned}
0\too H^1(S,\mathcal{H}om_{S-gr}(G,H))&\too \Ext^1_S(G,H)\too
H^0(S,\mathcal{E}xt^1_S(G,H))\too\\ \too &
H^2(S,\mathcal{H}om_{S-gr}(G,H))\too \Ext^{2}_S(G,H).
\end{aligned}
\end{equation}
Moreover $H^1(S,\mathcal{H}om_{S-gr}(G,H))$ is isomorphic to the
subgroup of $\Ext^1_S(G,H)$ formed by the extensions $E$ which
split over some faithfully flat affine $S$-scheme of finite type
(cf. \cite[III \S6, 3.6]{DG}). 
We will consider the case $H=\gmS$ and $G$ a finite flat group
scheme. In this case, $\mathcal{H}om_{S-gr}(G,\gmS)$ is   the
Cartier
dual of $G$, denoted by $G^\vee$. 
Since, by \cite[6.2.2]{Ra2}, $\mathcal{E}xt^1_{S}(G,\gmS)=0$ then
we obtain that the natural morphism
\begin{equation}\label{eq:ext1=H^1 per glbn}
H^{1}(S,G^\vee){\too} \Ext^1_S(G,{\gmS})
\end{equation}
of \eqref{eq:spectral sequence per ext} is in fact  an
isomorphism. 

\begin{prop}\label{prop:torsori schemi normali}
Let $X$ be a normal integral scheme. 
For any finite and flat commutative group scheme $G$ over $X$,
$$
i^*:\h{1}{X}{G}\too \h{1}{\Sp(K(X))}{G_{K(X)}}
$$
is injective, where $i:\Sp(K(X))\too X$ is the generic point.
\end{prop}
\begin{proof} A sketch of the proof has been suggested to us by F.
Andreatta. We recall that for any commutative group scheme $G'$
over a scheme $X'$ we have that $\h{1}{X'}{G'}$ is a group and it
classifies the $G'$-torsors over $X'$. Suppose there exists a
$G$-torsor $f:Y\too X$ such that $i^*f:i^*Y\to \Sp(K(X))$ is
trivial. This means there exists a section $s$ of $i^{*}f$. We
consider the scheme $Y_0$ which is the schematic closure of
$s(\Sp(K(X)))$ in $Y$. Then $f_{|Y_0}:Y_0\too X$ is a finite
birational morphism with $X$ a normal integral scheme. So, by
Zariski's Main Theorem (\cite[4.4.6]{liu}), we have that
$f_{|Y_0}$ is an open immersion, and so it is an isomorphism. So
we have a section of $f$ and $Y$ is a trivial $G$-torsor.
\end{proof}
\begin{rem}The hypothesis $G$ finite and $X$ normal   are necessary.
For the first it is sufficient to observe that  any
$\mathbb{G}_{m,X}$-torsor is trivial on
$\Sp(K(X))$. 
For the second one, consider $X=\Sp(k[x,y]/(x^p-y^{p+1}))=\Sp(A)$,
with $k$ any field of characteristic $p>0$, and $Y$ the
$\alpha_p$-torsor $\Sp(A[T]/(T^{p}-y))$. Generically this torsor
is trivial since we have $y=(\frac{x}{y})^p$. But $Y$ is not
trivial since $y$ is not a $p$-power in $A$.
\end{rem}
We deduce the following result which however it will not be  used
in the sequel.
\begin{cor}Let $X$ be a normal integral scheme.  Let $f:Y\too X$ be a
morphism with a rational section and let $g:G\too G'$ be a map of
finite and flat commutative group schemes over $X$, which is an
isomorphism over $\Sp(K(X))$. Then
$$
f^*g_*:\hxg{G}\too H^1(Y,G'_{Y})
$$
is injective.
\end{cor}
\begin{proof}
By hypothesis $\Sp(K(X))\too X$ factors through $f:Y\too X$. If
$i:\Sp(K(X))\too X$,  we have
$$
i_*:\hxg{G}\on{g_*}{\too} \hxg{G'}\on{f^*}{\too} H^1(Y,G'_{Y})\too
H^1(\Sp(K(X)),G_{K(X)}).
$$
Therefore, by the previous proposition, it follows that
$$
\hxg{G}\too H^1(Y,G'_{Y})
$$
is injective.
\end{proof}
\begin{rem}The previous corollary can be applied, for instance, to the case
$f=\id_X$ or to the case $f:U\too X$ an open immersion and
$g=\id_G$. Roberts (\cite[p. 692]{Rob}) has proved the corollary
in the case $f=\id_X$, with $X=\Sp(A)$ and $A$ the  ring of
integers in a local number field.
\end{rem}

\section{Models of $\mu_{p^2,K}$}\label{sec:modelli di Z/p^2Z}

First of all we prove that any  model of $\mu_{p^2,K}$ is an
abelian extension of $G_{\mu,1}$ by $G_{\lb,1}$ for some
$\mu,\lb\in R\setminus\{0\}$.
\begin{lem}\label{lem:modelli di Z/p^2 Z sono estensioni}
Let $G$ be a finite and flat $R$-group scheme of order $p^2$ such
that $G_K$ is diagonalizable. Then $G$ is an extension of
$G_{\mu,1}$ by $G_{\lb,1}$ for some $\mu,\lb\in R\setminus\{0\}$.
Moreover the generic fiber of any extension $G$ of $G_{\mu,1}$ by
$G_{\lb,1}$ is  multiplicative. In particular $G$ is commutative.
\end{lem}
\begin{proof}We have that $G_K$ is isomorphic to $\mu_{p^2,K}$ or to $\mu_{p,K}\times
\mu_{p,K}$. We consider the factorization
$$
0\too \mu_{p,K}\too G_K\too \mu_{p,K}\too 0.
$$
We take the schematic closure  $G_1$ of $\mu_{p,K}$ in $G$. Then
$G_1$ is a model of $\mu_{p,K}$. So by \ref{teo:modelli di Z/pZ}
it follows that $G_1\simeq G_{\lb,1}$ for some $\lb\in
R\setminus\{0\}$. Moreover $G/G_{\lb,1}$ is a model of
$\mu_{p,K}$, too. So, again by \ref{teo:modelli di Z/pZ}, we have
$G/G_{\lb,1}\simeq G_{\mu,1}$ for some $\mu\in R\setminus\{0\}$.
The first assertion is proved.

Now let $G$ be an extension of $G_{\mu,1}$ by $G_{\lb,1}$. We
recall that $G_K$ is of multiplicative type if and only if
$G_{\overline{K}}$ is diagonalizable. If $char(R)=0$ then
$G_{\overline{K}}$ is a constant group of order $p^2$. Since
$\overline{K}$ is algebraically closed,  it is diagonalizable. If
$char(R)=p$ then $G_{\overline{K}}$  is isomorphic to
$\mu_{p,\overline{K}}\times \mu_{p,\overline{K}}$ or
$\mu_{p^2,\overline{K}}$ by \cite[III \S 6, 8.7]{DG}.
\end{proof}

So it seems natural to study  the group
$\Ext^1_S(G_{\mu,1},G_{\lb,1})$. We observe that by the previous
lemma it follows that it is not restrictive to consider only
abelian extensions, as supposed at the beginning of
\S\ref{par:review of results}. Before beginning the study of
$\Ext^1_S(G_{\mu,1},G_{\lb,1})$ we prove the following lemma. 

\begin{lem}\label{lem:v(mu)>= v(lb)} Let $\clE$ be an extension of $G_{\mu,1}$ by $G_{\lb,1}$ with $\mu,\lb \in R\setminus\{0\}$. If $\clE$ is a
model of $\mu_{p^2,K}$ then necessarily $v(\mu)\ge v(\lb)$.
\end{lem}
\begin{proof}
Let us consider the map $p:\clE\too \clE$.  We observe that $p$ is
trivial on $G_{\lb,1}$, therefore we have an induced map
$p:G_{\mu,1}\too \clE$. Since $\clE_K\simeq \mu_{p^2,K}$ we have,
on the generic fiber, the factorization
$$
p:(G_{\mu,1})_K\too (G_{\lb,1})_K\hookrightarrow \clE_K,
$$
where the first map is an isomorphism. So $p$ induces a model map
$G_{\mu,1}\too G_{\lb,1}$, hence from \eqref{eq:Hom(glbn,glmn)} it
follows $v(\mu)\ge v(\lb)$.
\end{proof}
\subsection{Two exact sequences}\label{sec:two exact sequences}
The main tools which we will use to 
calculate the extensions of $\gmx{1}$ by $G_{\lb,1}$ are two exact
sequences. We recall them in this subsection. See \eqref{eq:ker
phi} and \eqref{eq:ker alpha} below.



Applying  the functor $\Ext$ to the following exact sequence of
group schemes
$$
(\Lambda): \qquad 0\too
\glx{1}\on{i}{\too}\glb\on{\psi_{\lb,1}}{\too} \g^{(\lb^{p})}\too
0,
$$
we obtain
\begin{equation}
\begin{aligned}\label{eq:ker phi}
0\too\Hom_{S-gr}(G_{\mu,1},\glbp)\on{\delta'}{\too}
\Ext^1_S(G_{\mu,1},\glx{1})\on{i_*}\too& \\ \too
\Ext^1_S(G_{\mu,1},\glb)
&\on{\psi_{{\lb,1}_*}}{\too}\Ext^1_S(\gmx{1},\glbp).
\end{aligned}
\end{equation}
We remark that $\delta'$ is injective  by
\eqref{eq:Hom(glbn,glmn)}.
The map
\begin{equation}\label{eq:def delta'}
\delta':\Hom_{S-gr}(G_{\mu,1},\glbp)\too
\Ext^1_S(G_{\mu,1},\glx{1})
\end{equation}
is defined by
$$
\sigma_j\longmapsto (\sigma_j)^*(\Lambda),
$$
where $(\sigma_j)^*(\Lambda)$ is the extension class of the group
scheme
\begin{equation}\label{eq:E(mu,lb,1,j)}
\clE^{(\mu,\lb;1,j)}:=\Sp(R[T_1,T_2]/(\frac{(1+\mu
T_1)^p-1}{\mu^p},\frac{(1+\lb T_2)^p-(1+\mu T_1)^j}{\lb^p}),
\end{equation}
with the maps
\begin{align*}
\glx{1}&\too \clE^{(\mu,\lb;1,j)}\\
   T_1&\longmapsto 0\\
   T_2&\longmapsto T
\end{align*}
and
\begin{align*}
\clE^{(\mu,\lb;1,j)}&\too\gmx{1}\\
   T&\longmapsto T_1
\end{align*}
The structure of group scheme on $ \clE^{(\mu,\lb;1,j)}$ is
inherited by the obvious closed immersion in the two-dimensional
group
scheme $ \clE^{(\mu,\lb;1)}\simeq \gmu\times \glb$. We observe that, by \eqref{eq:Hom(glbn,glmn)}, $\delta'$ is nontrivial if and only if $v(\mu)\ge pv(\lb)$. 
Moreover, if $v(\mu)=v(\lb)=0$, $\clE^{(\mu,\lb;1,j)}$ is
isomorphic to the group scheme
$$
\clE_{j,S}:=\Sp(R[T_1,T_2]/(T_1^p-1,\frac{T_2^p}{T_1^j}-1)),
$$
which is the kernel of the morphism $\psi^j:(\gmS)^2\too (\gmS)^2$
given by $(T_1,T_2)\too (T_1^p,T_1^{-j}T_2^p)$. If $j\neq 0,p$
then $\clE_{j,S}\simeq \mu_{p^2,S}$, otherwise it is isomorphic to
$\mu_{p,S}\times \mu_{p,S}$.

We now suppose $\lb \in \piR$. From the exact sequence
\eqref{eq:succ esatta glb in gm} 
we obtain  the following long exact sequence
\begin{equation}\label{eq:ker alpha}
\begin{array}{ll}
0\too &\Hom_{S-gr}(\gmx{1},\g^{(\lb)})\too
\Hom_{S-gr}(\gmx{1},\gmS)\on{r_{\lb}}{\too}
\Hom_{S_{\lb}-gr}(\gmx{1},\mathbb{G}_{m,S_{\lb}})\on{\delta}{\too}\\
&\too \Ext^1_S(\gmx{1},\g^{(\lb)})\on{\alpha^{(\lb)}_{*}}{\too}
\Ext^1_S(\gmx{1},\gm)\too
\Ext^1_{S_{\lb}}({\gmx{1}},{\mathbb{G}_{m,S_{\lb}}}).
\end{array}
\end{equation}

 So we   have, using \eqref{eq:Hom(glbn,glmn)},
\begin{equation}\label{eq:ker alpha con r lambda}
 \ker{\alpha^{(\lb)}_{*}}\simeq
\Hom_{S_{\lb}-gr}(\gmx{1},\mathbb{G}_{m,S_{\lb}})/\left<1+\mu
T\right>.
\end{equation}

 In the following we give a more explicit description of the main
ingredients of the exact sequences \eqref{eq:ker phi} and \eqref{eq:ker alpha}.

\subsection{Explicit description of
$\Hom_{S_{\lb}-gr}({\gmx{1}},{\gmSl}$)}\label{subsec:explicit
description...}

Let $\lb \in \pi R\setminus\{0\}$.  First we consider the simplest
case: by \cite[II \S 1, 2.11]{DG}
\begin{equation}\label{eq:Hom(mup,gm)su Slb}
\Hom_{S_{\lb}-gr}({\mu_{p,S_\lb}},{\gmSl})=\{T^i\in (R/\lb
R)[T,1/T];i\in \Z/p\Z\}.
\end{equation}

Now we study $\Hom_{S_{\lb}-gr}({\gmx{1}},{\gmSl})$ in general.
\begin{prop}\label{lem:suriettività mappa tra hom}
Let    $\mu,\lb \in \pi R\setminus\{0\}$. We  suppose
 $ (p-1)v(\mu)\le v(p)$  if $char(R)=0$.
\begin{itemize}
\item[(i)]
 The map
$$
i^*:\Hom_{S_\lb-gr}(\gmu,{\gmSl})\too
\Hom_{S_\lb-gr}({G_{\mu,1}},\gmSl),
$$
induced by
$$
i:G_{\mu,1}\hookrightarrow \gmu,
$$
is surjective.
\end{itemize}
Moreover, in the case $char(R)=0$, we also  assume $(p-1)v(\lb)\le
v(p)$. 
\begin{itemize}
\item[(ii)]
  We have the following isomorphism of groups
$$
(\xi^0_{R/\lb R})_p:(R/\lb R)^{{\fr}-\mu^{p-1}}
\too\Hom_{S_\lb-gr}({\gmx{1}},{\gmSl})
$$
given by
$$
a\longmapsto E_p(a,\mu; T).$$ Moreover $ E_p(a,\mu; T)=
1+\sum_{i=1}^{p-1}\frac{\prod_{k=0}^{i-1}(a-k\mu)}{i!}T^i$.

 \item[(iii)] The restriction map
$$
i^*:\Hom_{S_\lb-gr}({\gmu},{\gmSl})\simeq\widehat{W}(R/\lb
R)^{{\fr}-[\mu^{p-1}]}\too
\Hom_{S_\lb-gr}({\gmx{1}},{\gmSl})\simeq (R/\lb
R)^{{\fr}-\mu^{p-1}}
$$
is given, in terms of Witt vectors,  by
\begin{align*}
\bbf{a}=(a_0,a_1,\dots,0,0,0,\dots)&\longmapsto a_0+
\sum_{j=1}^{\infty}(-1)^j(\prod_{r=0}^{j-1}(\frac{p}{\mu^{p-1}})^{p^r})a_j
&
\text{ if } char(R)=0,\\
\bbf{a}=(a_0,a_1,\dots,0,0,0,\dots)&\longmapsto a_0 & \text{ if }
char(R)=p.
\end{align*}
\end{itemize}
\end{prop}
\begin{rem} If $char(R)=0$ and $v(\lb)\le v(p)-(p-1)v(\mu)$ then $i^*(\bbf{a})=a_0$. 
\end{rem}
\begin{proof}

(i) An element  of $\Hom_{S_\lb-gr}({\gmx{1}},{\gmSl})$  could be
represented by a polynomial
$$F(T)=\sum_{i=0}^{p-1}a_i T^i\in (R/\lb R)[T]$$ such that
\begin{itemize}
    \item[(a)] $F(U)F(V)\equiv F(U+V+\mu UV)\mod (\frac{(1+\mu U)^p-1}{\mu^p}, \frac{(1+\mu V)^p-1}{\mu^p})$
    \item[(b)] $F(T)$ is invertible in $(R/\lb R)[T]/(\frac{(1+\mu
    T)^p-1}{\mu^p})$.
\end{itemize}
The condition (a) implies that $F(U)F(V)=F(U+V+\mu UV)$, since all
terms of $F(U)F(V)$ and $F(U+V+\mu UV)$ have degree at most $p-1$
in $U$ and $V$. Moreover the finite group scheme $(G_{\mu,1})_k$
is a closed subgroup scheme of $\mathbb{G}_{a,k}$ and therefore
the Cartier dual $(G_{\mu,1}^{\vee})_k$ is infinitesimal. In
particular, we have
$\Hom_{k-gr}((G_{\mu,1})_k,\mathbb{G}_{m,k})=(G_{\mu,1}^{\vee})_k(k)=\{0\}$.
Then $F(T)\equiv 1 \mod \pi$. This implies that $F(T)$ is
invertible in $(R/\lb R)[T]$, since $\pi$ is nilpotent in $R/\lb
R$. Therefore we have that $F(T)$ represents an element of
$\Hom_{S_\lb-gr}({\gmu},{\gmSl})$ and the map $i^*$ is surjective.

(ii) By the exact sequence
$$
  \qquad 0\too G_{\mu,1}\on{i}{\too}
\gmu\on{\psi_{\mu,1}}{\too} \gmup\too 0
$$
over $S_\lb$, we have the long exact sequence of cohomology
\begin{equation}\label{eq:succ. esatta per hom  verso gm}
\begin{aligned}
0\too \Hom_{S_\lb-gr}(\gmup,{\gmSl})\on{\psi_{\mu,1}^*}{\too}
\Hom_{S_\lb-gr}(\gmu,{\gmSl})\on{i^*}{\too}\\
\too \Hom_{S_\lb-gr}({G_{\mu,1}},{\gmSl})\on{\delta''}{\too}&
\Ext^1_{S_\lb}(\gmup,{\gmSl})\too \dots
\end{aligned}
\end{equation}

We study separately three cases.

\begin{tabular}{|c|}
  \hline
$char(R)=0$ and     $v(\lb)\le (p-1)v(\mu)$. \\
 \hline
\end{tabular}
 From \ref{lem:somma termine per termine} we have that the
restriction of the Teichm\"{u}ller map
$$
T:(R/\lb R)^{\fr}\too \widehat{W}(R/\lb R)^{\fr},
$$
given by
$$
a\longmapsto [a],
$$
is a morphism of groups.
Moreover, if we consider the isomorphism
$$\xi^0_{R/\lb R}:\widehat{W}(R/\lb R)^{\fr}\too
\Hom_{S_\lb-gr}({\gmu},{\gmSl}),$$  
we have
$$
i^*\circ \xi^0_{R/\lb R}\circ T=(\xi^0_{R/\lb R})_p.
$$
So $(\xi^0_{R/\lb R})_p$ is a morphism of groups. We now prove
that it is surjective.

As remarked in (i)  any element of
$\Hom_{S_\lb-gr}({G_{\mu,1}},{\gmSl})$ can be represented by a
polynomial  $F(T)$ of degree at most $p-1$ with coefficients in
$R/\lb R$. By \cite[3.5, 3.7]{SS7} it satisfies $F(U)F(V)=
F(U+V+\mu UV)$ if and only if $F(T)=
1+\sum_{i=1}^{p-1}\frac{\prod_{k=0}^{i-1}(a-k\mu)}{i!}T^i$ for
some $a\in R/\lb R$ such that ${\prod_{k=0}^{p-1}(a-k\mu)}=0$.
But, since $p$ and $\mu^{p-1}$ are zero in $R/\lb R$, the last
condition is equivalent to $a^p=0$.  On the other hand, since
$a^p=\mu^p= 0\in R/\lb R$, then $a^i\mu^{p-i}=0$ for any
$i=0,\dots,p$. Therefore by \eqref{eq:E_p(a,mu;T)}, $E_p(a,\mu;T)$
is a polynomial  of degree at most $p-1$. Moreover it satisfies
$E_p(a,\mu;U)E_p(a,\mu;V)=E_p(a,\mu;U+V+\mu
 UV)$ and the coefficient of $T$ in
$E_p(a,\mu;T)$ is $a$. Hence
$E_p(a,\mu;T)=1+\sum_{i=1}^{p-1}\frac{\prod_{k=0}^{i-1}(a-k\mu)}{i!}T^i=F(T)$.


We now prove that  $(\xi^0_{R/\lb R})_p$ is injective. By
\eqref{eq:phi*:W(R)too W(R)}, \eqref{rem: tilde V per p=2},
\eqref{eq:psi per p=2},  and \eqref{eq:succ. esatta per hom verso
gm} its kernel is
$$T((R/\lb R)^{\fr})\cap \bigg\{
[\frac{p}{\mu^{p-1}}]\textit{\textbf{b}}+ V(\textit{\textbf{b}})
;\bbf{b} \in {\widehat{W}(R/\lb R)}^{\fr}\bigg\}.$$ Let us now
suppose that there exist $\bbf{b}=(b_0,b_1,\dots)\in
{\widehat{W}(R/\lb R)}^{\fr}$ and $ a\in (R/\lb R)^{\fr}$ such
that $[\frac{p}{\mu^{p-1}}]\textit{\textbf{b}}+
V(\textit{\textbf{b}})=[a]$.
 It follows by the definition of Witt vector ring that
\begin{equation}\label{eq:p mu^(p-1) b=...b_0 I}
[\frac{p}{\mu^{p-1}}]\textit{\textbf{b}}=(\frac{p}{\mu^{p-1}}b_0,
\dots, (\frac{p}{\mu^{p-1}})^{p^j} b_j,\dots),\end{equation} and
\begin{equation}\label{eq:a-V(b)}
[a]-V(\textit{\textbf{b}})=(a,-b_0,-b_1,\dots).
\end{equation}


Since $\bbf{b} \in {\widehat{W}(R/\lb R)} $, there exists $r\ge 0$
such that $b_j=0$ for any $j\ge r$. Moreover, comparing
\eqref{eq:p mu^(p-1) b=...b_0 I} and \eqref{eq:a-V(b)}, it follows
\begin{align*}
&(\frac{p}{\mu^{p-1}})^{p^{j+1}}b_{j+1}=-b_j \quad\text{ for } j\ge 0\\
&\frac{p}{\mu^{p-1}}b_0= a.
\end{align*}
Hence $b_j=a=0$ for any $j\ge 0$. It follows that $(\xi^0_{R/\lb
R})_p$ is injective.

\begin{tabular}{|c|}
  \hline
$char(R)=p$,  or    $char(R)=0$,  $p>2$ \text{ and }$(p-1)v(\mu)< v(\lb) .$   \\
 \hline
\end{tabular}
Since $p=0\in R/\lb R$ then $\fr V=V\fr$. Hence, using the fact
that $V$ is injective,  it is straightforward to prove that
$${V^{-1}\big(\widehat{W}(R/\lb
R)}^{\fr-[{\mu^{p-1}}])}\big)={\widehat{W}(R/\lb
R)}^{\fr-[{\mu^{p(p-1)}}]}.$$
Moreover if $char(R)=0$ then 
$$v(p)-(p-1)v(\mu)>v(p)-v(\lb)\ge (p-2)v(\lb)\ge v(\lb).$$
So it follows, using also  \eqref{eq:phi*:W(R)too W(R)},
\eqref{eq:psi per p=2} and \eqref{eq:succ. esatta per hom verso
gm},   that
\begin{align*}
(R/\lb R)^{\fr-[\mu^{p-1}]}\too \widehat{W}(R/\lb
R)^{\fr-[\mu^{p-1}]}/\bigg\{ V(\textit{\textbf{b}});\bbf{b} \in
{\widehat{W}(R/\lb R)}^{\fr-[{\mu^{p(p-1)}}]}\bigg\}\too
\Hom_{S_\lb-gr}({\gmx{1}},{\gmSl})
\end{align*}
are isomorphisms, where the first morphism is induced by the
Teichm\"{u}ller map, and the second one by $\xi^0_{R/\lb R}$. The
composition is nothing else but $(\xi^0_{R/\lb R})_p$.

Now let us consider $E_p(a,\mu;T)\in
\Hom_{S_\lb-gr}({\gmx{1}},{\gmSl})$. We observe that {$$T^p=0\in
(R/\lb R[T])/(\frac{(1+\mu T)^p-1}{\mu^p}).$$} If $char(R)=p$ this
is clear. While if $char(R)=0$ then
$T^p=-\frac{p}{\mu^{p-1}}(T+\frac{p-1}{2}\mu
T^2+\dots+\mu^{p-2}T^{p-1})$. But, by hypothesis,
$v(\frac{p}{\mu^{p-1}})> v(\lb)$. Therefore $E_p(a,\mu;T)$ is a
polynomial of degree $p-1$. Reasoning as in the previous case we
have that
$E_p(a,\mu;T)=1+\sum_{i=1}^{p-1}\frac{\prod_{k=0}^{i-1}(a-k\mu)}{i!}T^i=F(T)$.

\begin{tabular}{|c|}
  \hline
    $char(R)=0$,  $p=2$ and $v(\mu)< v(\lb). $   \\
 \hline
\end{tabular}
 Since $ T^2=-\frac{2}{\mu} T\in
R[G_{\mu,1}]$, it follows that $\mu T^2 =0 \in R/\lb
R[G_{\mu,1}]$. So  $a^2=\mu a \in R/\lb R$ implies, by
\eqref{eq:E_p(a,mu;T)},
$$
E_p(a,\mu;T)=E_p(aT)=1+ a T\in \Hom_{S_\lb-gr}(G_{\mu,1},\gmSl)
$$
Thus  $(\xi^0_{R/\lb R})_p(a)= 1+ a T$. We now prove that
$(\xi^0_{R/\lb R})_p$ is a morphism of groups. Let $a,b\in (R/\lb
R)^{\fr-[\mu]}$. We observe that, since $v(\mu)<v(\lb)$, this
implies $\mu |a^2$ and $\mu| b^2$. Hence $\mu| ab$.  Then
$$
(1+a T)(1+ bT)=1+(a+b)T+ ab T^2=1+ (a+b- \frac{2}{\mu} ab) T= 1+
(a+b) T.
$$
It is easy to check that the above morphism is in fact an
isomorphism.

 (iii) First of all we remark that for any $\bbf{a}=(a_0, \dots,
a_j, \dots)\in \widehat{W}(R/\lb R)^{\fr-[\mu^{p-1}]}$ we have
$$
\bbf{a}=\sum_{j=0}^{\infty}V^j([a_j]).
$$
It is clear that for any $a\in R/\lb R$ we have $i^*([a])=a$.
While, by \eqref{eq:phi*:W(R)too W(R)}, \eqref{eq:psi per p=2}
 and
\eqref{eq:succ. esatta per hom  verso gm}, it follows that for any
$\bbf{b}\in \widehat{W}(R/\lb R)^{\fr-[\mu^{p-1}]}$

\begin{align*}
i^*{V(\bbf{b})}=-i^*([\frac{p}{\mu^{p-1}}]\bbf{b}) &\text{ if }
char(R)=0, \\
i^*{V(\bbf{b})}=0 & \text{ if } char(R)=p.
\end{align*}
Hence, if $char(R)=0$,
$i^*({V^j(\bbf{b})})=(-1)^ji^*([\prod_{r=0}^{j-1}(\frac{p}{\mu^{p-1}})^{p^r}]\bbf{b})$
for any $j\ge 1$ and then it follows that
\begin{align*}
i^*(\bbf{a})&=i^*(\sum_{j=0}^{\infty}V^j([a_j]))\\
            &=\sum_{j=0}^{\infty}i^*(V^j([a_j]))\\
            &=a_0+\sum_{j=1}^{\infty}(-1)^j(\prod_{r=0}^{j-1}(\frac{p}{\mu^{p-1}})^{p^r})a_j.
\end{align*}
While, if $char(R)=p$, we have
$$
i^*(\bbf{a})=i^*(\sum_{j=0}^{\infty}V^j([a_j]))=i^*([a_0])=a_0.
$$

\end{proof}
    \subsection{Explicit description of $\delta$}\label{par:delta}
Let $\mu,\lb \in \piR$ and, if $char(R)=0$, $(p-1)v(\mu)\le v(p)$.
The map
$$
\delta:\Hom_{S_\lb-gr}({\gmx{1}},{\gmSl}){\too}
\Ext^1_S(\gmx{1},\g^{(\lb)})
$$
can also be explicitly described. 
  We  have the following
commutative diagram
$$
\xymatrix{\Hom_{S_\lb-gr}(\gmu,{\gmSl})\ar[r]^{i^*}\ar[d]^{\alpha}&
\Hom_{S_\lb-gr}({G_{\mu,1}},{\gmSl})\ar[d]^{\delta}\ar[r]&0\\
\Ext^1_S(\gmu,\g^{(\lb)})\ar[r]^{i^*}&\Ext^1_S(\gmx{1},\g^{(\lb)})}
$$
where the first horizontal map is surjective by
\ref{lem:suriettività mappa tra hom}(i). So, given  $$F(T)\in
\Hom_{S_\lb-gr}({\gmx{1}},{\gmSl}),$$ we can choose a representant
in  $\Hom_{S_\lb-gr}({\gmu},{\gmSl})$ which we denote again by
$F(T)$ for simplicity.  Let $\widetilde{F}(T)\in R[S]$ any its
lifting. Then $\delta$ is defined by
$$
F(T)\longmapsto
[\widetilde{\clE}^{(\mu,\lb;\widetilde{F})}]:=i^*([\clE^{(\mu,\lb;\widetilde{F})}])=i^*(\alpha(F(T))).
$$
We observe that $\widetilde{\clE}^{(\mu,\lb;\widetilde{F})}$ is
the subgroup scheme of $\clE^{(\mu,\lb;\widetilde{F})}$, defined
as a scheme by
$$
\widetilde{\clE}^{(\mu,\lb;\widetilde{F})}=\Sp\big(R[T_1,T_2,(\widetilde{F}(T_1)+\lb
T_2)^{-1}]/\frac{(1+\mu T_1)^p-1}{\mu^p}\big).
$$

%
The above extension does not depend on the choice of the lifting
since the same is true for $[{\clE}^{(\mu,\lb;\widetilde{F})}]$.
We have therefore proved the following proposition.
\begin{prop}\label{prop:ker alpha=clE}Let  $\mu,\lb\in \pi R\setminus\{0\}$ and, if $char(R)=0$, $(p-1)v(\mu)\le v(p)$.
Then $\delta$ induces an isomorphism 
\begin{align*}
 \Hom_{S_\lb-gr}({\gmx{1}},{\gmSl})/r_{\lb}(\Hom_{S-gr}(\gmx{1},\gmS))&\too  \{[\widetilde{\clE}^{(\mu,\lb;\widetilde{F})}];F\in \Hom_{S_\lb-gr}({\gmx{1}},{\gmSl})\}\\
F(T)&\longmapsto [\widetilde{\clE}^{(\mu,\lb;\widetilde{F})}]
\end{align*}
where $\widetilde{F}(T)$ is a lifting of $F(T)$.
\end{prop}


We finally remark that, by \eqref{eq:ker alpha}, we also have that
$\ker(\alpha^{(\lb)}_*)\In \Ext^1_S(\gmx{1},\glb)$ is nothing else
but the group $$ \{[\widetilde{\clE}^{(\mu,\lb;\widetilde{F})}]
;F\in \Hom_{S_\lb-gr}({\gmx{1}},{\gmSl})\}.$$




\subsection{Description of $\Ext^1_S(\gmx{1},\glx{1})$}\label{par:Ext1(gmu,glb)}

We finally have all the ingredients to  give a description of the
group $\Ext^1_S(\gmx{1},\glx{1})$. In the rest of the section we
will suppose, if $char(R)=0$, that $(p-1)v(\mu)\le v(p)$ and
$(p-1)v(\lb)\le v(p)$. We distinguish, for clarity, four cases
$$
\begin{array}{ll}
    \Ext^1_S(\mu_{p,S},\mu_{p,S}); \\
    \Ext^1_S(\gmx{1},\mu_{p,S}), & \hbox{with $\mu \in \pi R \setminus\{0\}$;} \\
    \Ext^1_S(\mu_{p,S},\glx{1}), & \hbox{with  $\lb\in \pi R\setminus\{0\}$;} \\
    \Ext^1_S(\gmx{1},\glx{1}), & \hbox{with $\mu,\lb \in \pi R \setminus\{0\}$.} \\
\end{array}%
$$


The first  three cases are easy. The first two cases have already
been treated in \cite{SS8} and the third one can be  obtained with
an argument identical to that one used for the proof of the first
case. We report here the proofs.
\begin{prop}\label{eq:ext1(mup,mup)}
We have the following exact sequences. 
\begin{itemize}
    \item[(i)] $0\too\Z/p\Z\too \Ext^1_{S}(\mu_{p,S},\mu_{p,S})\too
H^1(S,\Z/p\Z)\too 0$;
    \item[(ii)] if $v(\mu)>0$,  $0\too\Z/p\Z\too \Ext^1_{S}(\gmx{1},\mu_{p,S})\too
H^1(S,G_{\mu,1}^\vee)\too 0$;
    \item[(iii)] if $v(\lb)>0$, $0\too \Ext^1_{S}(\mu_{p,S},\glx{1})\too
    H^1(S,\Z/p\Z)\too H^1(S_{\lb},\Z/p\Z)$.
\end{itemize}
\end{prop}

%
\begin{proof}
We set $G=G_{\mu,1}$, with $v(\mu)>0$, or $G=\mu_{p,S}$. The
Kummer sequence
$$
1\too \mu_{p,S}\too \gmS\on{p}{\too} \gmS\too 1
$$
yields, using the isomorphism \eqref{eq:ext1=H^1 per glbn},  an
exact sequence
$$
0 {\too} \Hom_{S-gr}(G,\gmS)\too \Ext^1_S(G,\mu_{p,S})\too
H^1(S,G^\vee)\on{p}{\too} H^1(S,G^\vee).
$$
This is a particular case of \eqref{eq:ker phi}. Since $G$ is
annihilated by $p$ then we have the exact sequence
\begin{equation}\label{eq:succ esatta particolare} 0{\too} \Hom_{S-gr}(G,\gmS)\too \Ext^1_S(G,\mu_{p,S})\too
H^1(S,G^\vee)\too 0.
\end{equation}
  The proofs of (i) and (ii) follow,
since, by \eqref{eq:Hom(glbn,glmn)}, we have
$$
\Hom_{S-gr}(G_{\mu,1},\gmS)\simeq
\Hom_{S-gr}(\mu_{p,S},\gmS)\simeq \Z/p\Z.
$$
We now prove (iii).
   Combining
the isomorphism \eqref{eq:ext1=H^1 per glbn} and the exact
sequence \eqref{eq:ker alpha} we obtain the exact sequence
\begin{align*}
 \Hom_{S-gr}(G_{\mu,1},\gmS)\on{r_{\lb}}{\too}
 \Hom_{S_\lb-gr}(G_{\mu,1},\gmSl)&\on{\delta}{\too}\\
\too \Ext^1_S(G_{\mu,1},\glb)\too& H^1(S,G_{\mu,1}^{\vee}) \too
H^1({S_\lb},G_{\mu,1}^{\vee}).
\end{align*}
If $\mu$ is invertible in $R$  then  $G_{\mu,1}\simeq \mu_{p,S}$
and we obtain
$$
 \Ext^1_S(\mu_{p,S},\glb)\simeq \ker ( H^1(S,\Z/p\Z)
\too H^1({S_\lb},\Z/p\Z))
$$
since the reduction map $r_\lb:
\Hom_{S-gr}(\mu_{p,S},\gmS)\on{}{\too}
 \Hom_{S_\lb-gr}(\mu_{p,S_\lb},\gmSl)$ is surjective.  On the other hand let us  consider the
 commutative diagram
\begin{equation*}
\xymatrix{
\Ext^1_S(\mu_{p,S},\glb)\ar^{\alpha_*^{(\lb)}}[d]\ar^{\psi_{{\lb,1}_*}}[r]&
\Ext^1_S(\mu_{p,S},\glbp)\ar^{\alpha_*^{(\lb^p)}}[d]\\
     \Ext^1_S(\mu_{p,S},\gmS)\ar^{p_*}[r]&
\Ext^1_S(\mu_{p,S},\gmS)  }    \end{equation*} The  morphism $p_*$
is trivial since $\Ext^1$ is a bi-additive functor and $\mu_{p,S}$
has order $p$.
 From \eqref{eq:ker
alpha con r lambda} it follows that
$\alpha^{(\lb^p)}_*:\Ext^1_S(\mu_{p,S},\glbp)\too
\Ext^1_S(\mu_{p,S},\gmS)$ is injective, therefore
$\psi_{{\lb,1}_*}$ is the zero map. Now, by
\eqref{eq:Hom(glbn,glmn)} we have $\Hom_{S-gr}(\mu_{p,S},\glb)=0$;
 hence, from  the exact sequence \eqref{eq:ker alpha} it follows
$$\Ext^1_S(\mu_{p,S},\glb)\simeq \Ext^1_S(\mu_{p,S},\glx{1}).$$
So we are done.
\end{proof}
\begin{rem}\label{rem:Ext(mup,mup) ogni base} The proof of (i)  works over any base, not necessarily
a d.v.r.
\end{rem}

\begin{rem}\label{rem: clE sono lo Z/p} 
By \S\ref{sec:two exact sequences} the group $\Z/p\Z$ which
appears in the exact sequence (i) corresponds to the group of
extensions $\{[\clE_{j,S}]; j=0,\dots,p-1\}$ and that one of
 (ii) corresponds to
$\{[\clE^{(\mu,\lb;1, j)}]; j=0,\dots,p-1\}$, with $v(\mu)>0$ and
$v(\lb)=0$.

\end{rem}



\begin{lemdefn}\label{def:clE(mu,lb,F,i)}
Let $\mu,\lb\in R\setminus\{0\}$. Let us assume $F(T)\in (R/\lb
R)[T]$, $j\in \Z$ and
\begin{itemize}
    \item [(a)]   $F(0)=1$,
    \item [(b)]   $F(U)F(V)= F(U+V+\mu UV)$,
    \item [(c)]   $F(T)^p\equiv (1+\mu T)^j \mod (\lb^p, \frac{(1+\mu
    T)^p-1}{\mu^p})$.
\end{itemize}
Let $\widetilde{F}(T)\in R[T]$ be any its lifting.  Then
$$
\clE^{(\mu,\lb;\widetilde{F},j)}:=\Sp\bigg(R[T_1,T_2]/\big(\frac{(1+\mu
T_1)^p-1}{\mu^p},\frac{({\widetilde{F}}(T_1)+\lb T_2)^p(1+\mu
T_1)^{-j}-1}{\lb^p}\big)\bigg),
$$
is a closed subgroup scheme of $\clE^{(\mu,\lb;\widetilde{F})}$.
\end{lemdefn}
\begin{proof}The proof is straightforward.
\end{proof}
\begin{rem}If $F(T)=1$ the definition coincides with
\eqref{eq:E(mu,lb,1,j)}.
\end{rem}
\begin{rem}In the above definition the integer $j$ is uniquely
determined by $F(T)$ if and only if $\lb^p\nmid \mu$.
\end{rem}
\begin{rem}\label{rem:F in Hom}
Let $\lb \in \piR$. We stress that  (a) and (b) means that
$F(T)\in \Hom_{S_\lb-gr}(\gmu,{\gmSl})$.
\end{rem}

\vspace{.5cm}

The closed immersion
$\clE^{(\mu,\lb;\widetilde{F},j)}\hookrightarrow
\clE^{(\mu,\lb;\widetilde{F})}$ induces the following commutative
diagrams of exact rows:
\begin{equation*}
\begin{aligned}
\xymatrix@1{0\ar[r]&{G_{\lb,1}}\ar[r]\ar[d]&\clE^{(\mu,\lb,\widetilde{F},j)}\ar[d]\ar[r]&G_{\mu,1}\ar[r]\ar[d]&0\\
                     0\ar[r]&{\glb}\ar[r]&\clE^{(\mu,\lb,\widetilde{F})}\ar[r]&\gmu \ar[r]&0}
\end{aligned}
\end{equation*}
and
\begin{equation*}
\begin{aligned}
\xymatrix@1{0\ar[r]&{G_{\lb,1}}\ar[r]\ar^{\alpha^{(\lb)}}[d]&\clE^{(\mu,\lb;\widetilde{F},j)}\ar^{\alpha^{(\mu,\lb,\widetilde{F})}}[d]\ar[r]&G_{\mu,1}\ar[r]\ar^{\alpha^{(\mu)}}[d]&0\\
                     1\ar[r]&{\muS}\ar[r]&\clE_{j,S}\ar[r]&\muS \ar[r]&1}
\end{aligned}
\end{equation*}
where $\alpha^{(\mu,\lb,\widetilde{F})}$ is a model map. Since the
 extension class does not depend on the choice of the lifting $\tilde{F}$, in the
following we will sometimes denote it  simply with
$[\clE^{(\mu,\lb;{F},j)}]$, omitting to specify the lifting. We
will use the same convention also for the extensions
$[\widetilde{\clE}^{(\mu,\lb;{F})}]$.

\begin{lem} Let $\mu,\lb \in R\setminus\{0\}$ and let $F(T)\in R/\lb R[T]$ be as in the lemma-defintion.
The group scheme $\clE^{(\mu,\lb;\widetilde{F},j)}$  is the kernel
of the isogeny
\begin{align*}
\psi^j_{\mu,\lb,\widetilde{F},\widetilde{G}}:\clE^{(\mu,\lb;\widetilde{F})}&\too \clE^{(\mu^p,\lb^p;\widetilde{G})}\\
          T_1&\longmapsto \frac{(1+\mu T_1)^p-1}{\mu^p}\\
          T_2&\longmapsto \frac{(\widetilde{F}(T_1)+\lb
T_2)^p(1+\mu T_1)^{-j}-\widetilde{G}(\frac{(1+\mu
T_1)^p-1}{\mu^p})}{\lb^p}
\end{align*}
for some $G(T)\in R/\lb^{p} R[T]$ which satisfies $(a),(b),(c)$ of
lemma-definition \ref{def:clE(mu,lb,F,i)}; and
$\widetilde{F}(T),\widetilde{G}(T)\in R[T]$ are liftings of $F(T)$
and $G(T)$ respectively.
\end{lem}

 \begin{proof}
If $v(\lb)=0$ we can assume $\tilde{F}(T)=1$. Therefore if we take
$\tilde{G}=1$ we are done. Let us assume $\lb \in \piR$. By
\ref{rem:F in Hom} we have  $F(T)\in
\Hom_{S_\lb-gr}(\gmu,{\gmSl})$. If we restrict this morphism to
$G_{\mu,1}$ we obtain an element of
$\Hom_{S_\lb-gr}({G_{\mu,1}},{\gmSl})$, which we denote again
$F(T)$ for simplicity. Now the condition (c) means $F(T)^p(1+\mu
T)^{-j}=1\in
\Hom_{S_{\lb^p}-gr}({\gmx{1}},{\mathbb{G}_{m,S_{{\lb^p}}}})$. Let
us consider the exact sequence over $S_{\lb^p}$
$$ 0\too G_{\mu,1}\on{i}{\too}
\gmu\on{\psi_{\mu,1}}{\too}\gmup\too 0.
$$
Applying to this exact sequence the functor
$\Hom_{S_{\lb^p}-gr}(\cdot,{\gmSlp})$ we obtain
\begin{equation}\label{eq:ker i_*}
\ker \bigg(i^*:\Hom_{S_{\lb^p}-gr}(\gmu,{\gmSlp})\too
\Hom_{S_{\lb^p}-gr}({G_{\mu,1}},{\gmSlp})\bigg)={\psi_{\mu,1}}_*\Hom_{S_{\lb^p}-gr}(\gmup,{\gmSlp}).
\end{equation}

Therefore condition (c) is equivalent to saying that there
exists $G(T)\in\Hom_{S_{\lb^p}-gr}({\gmup},{\gmSlp})$ with the
property that $F(T)^p(1+\mu T)^{-j}=G(\frac{(1+\mu
T)^p-1}{\mu^p})\in \Hom_{S_{\lb^p}-gr}({\gmu},{\gmSlp})$. This
implies the thesis. 

\end{proof}

We observe that we have the following commutative diagram of exact
rows
\begin{equation*}
\begin{aligned}
\xymatrix@1{
0\ar[r]&\clE^{(\mu,\lb,\widetilde{F},j)}\ar[d]\ar[r]&\clE^{(\mu,\lb,\widetilde{F})}\ar_{\alpha^{(\mu,\lb,\widetilde{F})}}[d]\ar^{\psi^j_{\mu,\lb,\widetilde{F},\widetilde{G}}}[r]&\ar[r]\ar^{\alpha^{(\mu,\lb,\widetilde{G})}}[d]\clE^{(\mu^p,\lb^p,\widetilde{G})}&0\\
                     0\ar[r]&\clE_{j,S}\ar[r]&(\gmS)^2\ar^{\psi^j}[r]&(\gmS)^2
                     \ar[r]&0}
\end{aligned}
\end{equation*}
 and the
following commutative diagram of exact rows and columns
\begin{equation*}
\begin{aligned}
\xymatrix@1{ & 0\ar[d]& 0\ar[d]& 0\ar[d]\\
0\ar[r]&{G_{\lb,1}}\ar[r]\ar[d]&\clE^{(\mu,\lb,\widetilde{F},j)}\ar[d]\ar[r]&G_{\mu,1}\ar[r]\ar[d]&0\\
                     0\ar[r]&{\glb}\ar[r]\ar^{\psi_{\lb,1}}[d]&\clE^{(\mu,\lb,\widetilde{F})}\ar[r]\ar^{\psi^j_{\mu,\lb,\widetilde{F},\widetilde{G}}}[d]&\gmu
                     \ar[r]\ar^{\psi_{\mu,1}}[d]&0\\
0\ar[r]&{\glbp}\ar[r]\ar[d]&\clE^{(\mu^p,\lb^p,\widetilde{G})}\ar[r]\ar[d]&\gmup
                     \ar[r]\ar[d]&0\\
                     &0 &0 &0}
\end{aligned}
\end{equation*}


 \begin{ex}\label{ex:Z/p^2Z SS}
 Let us suppose that $R$ is a d.v.r of unequal characteristic which contains a primitive $p^2$-th root of unity $\zeta_{2}$.  We define
$$
\eta=\sum_{k=1}^{p-1}\frac{(-1)^{k-1}}{k}\lb_{(2)}^k.
$$
We remark that $v(\eta)=v(\lb_{(2)})$. We put
$$
\widetilde{F}(T)=\sum_{k=1}^{p-1}\frac{(\eta T)^k}{k!}
$$
We have  $\widetilde{F}(T)\equiv E_p(\eta T)\mod \lb_{(1)}$. It
has been shown in \cite[\S 5]{SS4} that, using our notation,
$$
\Z/p^2\Z\simeq \clE^{(\lb_{(1)},\lb_{(1)};\tilde{F}(T),1)}.
$$
A similar description of $\Z/p^2\Z$ was independently found by
Green and Matignon (\cite{GM1}). This  gives the explicit
Kummer-Artin-Schreier-Witt theory for cyclic covers of order
$p^2$.
\end{ex}

\begin{ex}\label{ex:Gmun}
Assume $\mu,\lb \in R\setminus \{0\}$ and  $j\not\equiv 0 \mod p$.
Then there exists a group scheme  of type
$\clE^{(\mu,\lb;\tilde{F},j)}$,
with $\tilde{F}(T)\equiv 1\mod \lb$, if and only if 
 $v(\mu)\ge pv(\lb)$. Indeed 
the condition  $\tilde{F}(T)^p\equiv (1+\mu T)^j\mod (\lb^p,
\frac{(1+\mu T)^p-1}{\mu^p})$  reads  as $1 \equiv (1+\mu T)^j\mod
(\lb^p,
\frac{(1+\mu T)^p-1}{\mu^p})$. This is the case if and only if $\mu \equiv 0 \mod \lb^p$. 

In the limit case $v(\mu)=pv(\lb)$  we have an isomorphism
$$\clE^{(\lb^p,\lb;1,1)}\too
G_{\lb,2}:=\Sp(R[T]/(\frac{(1+\lb T)^{p^2}-1}{\lb^{p^2}})\In \glb,
$$
given, on the level of Hopf algebras, by
$$
T\mTo T_2
$$
and the inverse by
$$
T_1\mTo \frac{(1+\lb T)^p-1}{\lb^{p}},  \qquad T_2\mTo T.
$$
We observe that, if $char(R)=0$, since $(p-1)v(\mu)\le v(p)$, then
$v(\mu)=pv(\lb)$ implies $p(p-1)v(\lb)\le v(p)$. The last one is
in fact the  condition which  ensures that the flat $R$-group
scheme $G_{\lb,2}$ is finite over $R$.
\end{ex}

\begin{defn}
Let $\mu \in R\setminus\{0\}$ and $\lb\in \pi R\setminus \{0\}$.
We define
\begin{align*}
\rad&:=\bigg\{(F(T),j)\in \Hom_{S_\lb-gr}({\gmx{1}},{\gmSl})\times
\Z/p\Z \text{
such that }\\
& F(T)^p(1+\mu T)^{-j}=1\in
\Hom_{S_{\lb^p}-gr}({\gmx{1}},{\mathbb{G}_{m,S_{\lb^p}}})\bigg\}/
\left<(1+\mu T,0) \right>.
\end{align*}
It is a subgroup of $(\Hom_{S_\lb-gr}({\gmx{1}},{\gmSl})\times
\Z/p\Z)/\left<(1+\mu T,0) \right>$. If $v(\lb)=0$ we set
$\rad:=\{1\}\times \Z/p\Z$.

\end{defn}
\begin{rem}\label{rem: rad=0}
If $v(\mu)=0$ and $v(\lb)>0$ then the  group is trivial by
\eqref{eq:Hom(mup,gm)su Slb}.
\end{rem}
\vspace{.5cm}

Let $\mu,\lb \in R\setminus\{0\}$. We define $$ \beta: \rad
\on{}{\too} \Ext^1_S(G_{\mu,1},\glx{1})$$ by
$$ (F(
T),j)\longmapsto [\clE^{(\mu,\lb;{F}(T),j)}].
$$
 It is easy to
see that it is well defined.
\begin{lem}
$\beta$ is a morphism of groups. In particular the set
$\{[\clE^{(\mu,\lb;F,j)}]; (F,j)\in \rad\}$ is a subgroup of
$\Ext^1_S(G_{\mu,1},\glx{1})$.
\end{lem}
\begin{proof} If $v(\mu)=0$ and $v(\lb)>0$ the statement is trivial. If $v(\lb)=0$ then $\beta$ is
the morphism $\Z/p\Z\too \Ext^1_S(G_{\mu,1},\mu_{p,S})$ of
\ref{eq:ext1(mup,mup)}(i),(ii) (see also \ref{rem: clE sono lo
Z/p}).

Now let us assume $\mu,\lb \in \piR$. Let $i:G_{\lb,1}\too \glb$.
We remark that $$ i_*(\beta(F,j))=i_*(
[\clE^{(\mu,\lb;F,j)}])=[\widetilde{\clE}^{(\mu,\lb;F)}]=\delta(F)$$
for any $(F,j)\in \rad$. Moreover  by construction
$$
[(\clE^{(\mu,\lb;\widetilde{F},j)})_K]=[(\cal{E}_{j})_K]\in
\Ext^1_S({\mu_p}_{,K},{\mu_p}_{,K}).
$$

 Let $(F_1,j_1),(F_2,j_2)\in \rad$. Then
\begin{equation}\label{eq:i_*=0}
i_*(\beta(F_1,j_1)+\beta(F_2,j_2)-\beta(F_1+F_2,j_1+j_2))=\delta(F_1)+\delta(F_2)-\delta(F_1+F_2)=0,
\end{equation}
since $i_*$ and $\delta$ are morphisms of groups. For an extension
$[\clE]\in \Ext^1_S(G_{\mu,1},G_{\lb,1})$ we  define
$[\clE]_K:=[\clE_K]\in \Ext^1_K(\mu_{p,K}, \mu_{p,K})$. Then
\begin{equation}\label{eq:restrizione su K e' morfismo}
(\beta(F_1,j_1)+\beta(F_2,j_2)-\beta(F_1+F_2,j_1+j_2))_K=[\clE_{j_1,K}]+[\clE_{j_2,K}]-[\clE_{j_1+j_2,K}]=0.
\end{equation}
The last equality follows from \ref{eq:ext1(mup,mup)}(i) and
\ref{rem: clE sono lo Z/p}. By \eqref{eq:i_*=0} we have
$$
\beta(F_1,j_1)+\beta(F_2,j_2)-\beta(F_1+F_2,j_1+j_2)\in\ker i_*.
$$
and therefore,  by \eqref{eq:ker phi} and \eqref{eq:def delta'},
$$
\beta(F_1,j_1)+\beta(F_2,j_2)-\beta(F_1+F_2,j_1+j_2)=(\sigma_j)^*\Lambda.
$$
for some $j\in \Z/p\Z$. By \eqref{eq:restrizione su K e' morfismo}
it follows that
$$
((\sigma_j)^*\Lambda)_K=[\clE_{j,K}]=0,
$$
therefore $j=0$. So $\beta$ is a morphism of groups.
The last assertion is clear.
\end{proof}
 We  now give a description of $\Ext^{1}_S(G_{\mu,1},\glx{1})$ with $\mu,\lb \in \pi
 R\setminus\{0\}$. We recall that, till the end of the section, we are assuming that
$(p-1)v(\mu),(p-1)v(\lb)\le v(p)$ if
 $char(R)=0$.
\begin{thm}\label{teo:ext1(glx,gmx)}
Let $\mu,\lb \in \piR$.
 The 
following sequence
\begin{equation*}
\begin{array}{ll}
0\too \rad \on{\beta}{\too}&
\Ext^1_S(G_{\mu,1},\glx{1})\on{\alpha^{(\lb)}_*\circ i_*}\too\\
&\too \ker \bigg(H^1(S,G_{\mu,1}^\vee)\too
H^1(S_\lb,G_{\mu,1}^\vee)\bigg)
\end{array}
\end{equation*}
is exact. In particular $\beta$ induces an isomorphism
$$\rad\simeq \{[\clE^{(\mu,\lb;F,j)}];
(F,j)\in \rad\}.$$
\end{thm}
\begin{proof}
Using \eqref{eq:ker alpha} and  \ref{prop:ker alpha=clE}, we consider the following commutative diagram 
\begin{equation}\label{eq:succ esatta per ext1} \xymatrix{0\ar[r]
&\{\widetilde{\clE}^{(\mu,\lb;F)}\}\ar[r]\ar^{\widetilde{\psi_{{\lb,1}_*}}}[d]&
\Ext^1_S(\gmx{1},\glb)\ar^{\alpha_*^{(\lb)}}[r]\ar^{\psi_{{\lb,1}_*}}[d]&\Ext^1_S(\gmx{1},\gmS)\ar^{p_*}[d]\ar[r]&  \Ext^1_{S_{\lb}}({G_{\mu,1}},{\gmSl})\ar^{p_*}[d]\\
 0\ar[r]  &\{\widetilde{\clE}^{(\mu,\lb^p;G)}\}\ar[r]&
\Ext^1_S(\gmx{1},\g^{(\lb^p)})\ar^{{\alpha_*^{(\lb^p)}}}[r]
&\Ext^1_S(\gmx{1},\gmS)\ar[r]      &
\Ext^1_{S_{\lb}}({G_{\mu,1}},{\gmSl})}
\end{equation}
 The map
$\widetilde{\psi_{{\lb,1}_*}}$, 
induced by ${\psi_{\lb,1}}_*:\Ext^1_S(\gmx{1},\glb)\too
\Ext^1_S(\gmx{1},\glbp)$, is given by
$$[\widetilde{\clE}^{(\mu,\lb;\widetilde{F})}]\longmapsto
[\widetilde{\clE}^{(\mu,\lb^p;\widetilde{F}^p)}].$$ 
The functor $\Ext^1$ is bi-additive and $\gmx{1}$ has order $p$,
then \mbox{$p_*:\Ext^1_S(\gmx{1},\gmS)\to
\Ext^1_S(\gmx{1},\gmS)$} 
is the zero map.
Moreover, by \eqref{eq:spectral sequence per ext} and the
isomorphism \eqref{eq:ext1=H^1 per glbn}, we have the following
situation
$$
\xymatrix@1{0\ar[r]&H^1(S,G_{\mu,1}^\vee)\ar[r] \ar[d]& \Ext^1_S(\gmx{1},\gm) \ar[d]^{}\ar[r]&0\\
0\ar[r]&H^1(S_\lb,G_{\mu,1}^\vee)\ar[r]&\Ext^1_{S_{\lb}}({\gmx{1}},{\gmSl})}
$$
which
 implies that
$ \im(\alpha^{(\lb)}_*)\simeq \ker
(H^1(S,G_{\mu,1}^\vee)\too H^1(S_\lb,G_{\mu,1}^\vee))$. 
%
So applying the 
snake lemma to  \eqref{eq:succ esatta per ext1} we obtain 
\begin{equation}\label{eq:snake}
\begin{array}{l}
0\too \ker(\widetilde{\psi_{{\lb,1}_*}})
\on{\widetilde{\delta}}{\too}
\ker(\psi_{{\lb,1}_*})\on{\alpha^{(\lb)}_*} \too \ker
\bigg(H^1(S,G_{\mu,1}^\vee)\too
H^1(S_\lb,G_{\mu,1}^\vee)\bigg). 
\end{array}
\end{equation}
We now divide the proof in some steps.

\begin{tabular}{|c|}
  \hline
  Connection between $\ker(\widetilde{{\psi_{\lb,1}}_*})$ and $\rad$.\\
  \hline
\end{tabular}
We are going to give the connection in the form of the isomorphism
\eqref{eq:isomorfismo} below. We recall that, by \eqref{eq:ker
phi}, $i:\glx{1}\too \glb$ induces an isomorphism
\begin{equation}\label{eq:ker phi bis}
i_*:\Ext^1_S(G_{\mu,1},\glx{1})/\delta'(\Hom_{S-gr}
(G_{\mu,1},\glbp) )\too \ker({\psi_{\lb,1}}_*);
\end{equation}
for the definition of $\delta'$ see \eqref{eq:def delta'}.
By  \ref{prop:ker alpha=clE} we have an isomorphism 
\begin{align*}
    \delta:\Hom_{S_\lb-gr}({\gmx{1}},{\gmSl})/r_{\lb}(\Hom_{S-gr}(\gmx{1},\gmS))\too \{[\widetilde{\clE}^{(\mu,\lb;F)}];
F\in \Hom_{S_\lb-gr}({\gmx{1}},{\gmSl})\}.
\end{align*}
Through this identification we can identify 
$\ker(\widetilde{{\psi_{\lb,1}}_*})$ with

$$\left\{F(T)\in
\Hom_{S_{\lambda}-{{gr}}}(G_{\mu,1},\mathbb{G}_{m,S_{\lambda}})\;;
\begin{array}{c}
\text{there exists $ i\in r_{\lb^p}(\Hom_{S-gr}(\gmx{1},\gmS))$ such that}\\[1mm]
F(T)^p(1+\mu T)^{-i}=1\in
\Hom_{S_{\lambda^p}-{\mathrm{gr}}}(G_{\mu,1},\
G_{m,S_{\lambda^p}})
\end{array}\right\}\Bigl/\langle 1+\mu T \rangle$$
and 
\begin{equation*}\label{eq:inj}
\widetilde{\delta}:\ker(\widetilde{{\psi_{\lb,1}}_*})\hookrightarrow
\Ext^1_S(G_{\mu,1},\glx{1})/\delta'(\Hom_{S-gr} (G_{\mu,1},\glbp)
)\In \Ext^1_S(\gmx{1},\g^{(\lb)})
\end{equation*}
 is defined by 
$\widetilde{\delta}(F)=\delta(F)=[\widetilde{\clE}^{(\mu,\lb;{F})}]$. 
We now define  a morphism of groups
$$\iota:\ker(\widetilde{{\psi_{\lb,1}}_*})\too
r_{\lb^p}(\Hom_{S-gr}(\gmx{1},\gmS))$$ as follows: for  any
$F(T)\in \ker(\widetilde{{\psi_{\lb,1}}_*})$,  $\iota(F)=i_F$  is
the unique $i\in r_{\lb^p}(\Hom_{S-gr}(\gmx{1},\gmS))$ such that
$F(T)^p(1+\mu T)^{-i}=1\in
\Hom_{S_{\lb^p}-gr}({\gmx{1}},{\mathbb{G}_{m,S_{{\lb^p}}}})$. The
morphism of groups
\begin{align}\label{eq:isomorfismo}
\ker(\widetilde{\psi_{{\lb,1}_*}})\times \Hom_{S-gr}
(G_{\mu,1},\glbp)&\too \rad\\
(F,j)&\longmapsto (F,i_F+j)\nonumber
\end{align}
is  an isomorphism. We prove only the surjectivity  since the
injectivity is clear. Now, if $\lb^p\nmid \mu$ then
$\Hom_{S-gr}(\gmx{1},\glbp)=0$ and
$r_{\lb^p}(\Hom_{S-gr}(\gmx{1},\gmS))=\Z/p\Z$. So, if $(F,j)\in
\rad$, then $j\in r_{\lb^p}(\Hom_{S-gr}(\gmx{1},\gmS))$. So
$i_F=j$. Hence $(F,0)\mapsto (F,i_F)=(F,j)$. While if $\lb^p\mid
\mu$ then $\Hom_{S-gr}(\gmx{1},\glbp)=\Z/p\Z$ and
$r_{\lb^p}(\Hom_{S-gr}(\gmx{1},\gmS))=0$. Hence
$$\ker(\widetilde{{\psi_{\lb,1}}}_*)=\bigg\{F(T)\in
\Hom_{S_\lb-gr}({\gmx{1}},{\gmSl});
 F(T)^p=1\in
\Hom_{S_{\lb^p}-gr}({\gmx{1}},{\mathbb{G}_{m,S_{{\lb^p}}}})\bigg\}.$$
Let us now take $(F,j)\in \rad.$ This means that
$$F(T)^p=(1+\mu T)^j=1\in \Hom_{S_{\lb^p}-gr}({\gmx{1}},{\mathbb{G}_{m,S_{{\lb^p}}}}).$$
Therefore $F(T)\in \ker(\widetilde{\psi_{{\lb,1}_*}})$ and
$i_F=0$. So
$$
(F,j)\longmapsto (F,i_F+j)=(F,j).
$$

\begin{tabular}{|c|}
  \hline
  Interpretation of $\beta$.\\
  \hline
\end{tabular}
We now define the morphism of groups
\begin{align*}
\varrho:\ker(\widetilde{{\psi_{\lb,1}}_*})&\too
\Ext^1_S(G_{\mu,1},\glx{1})\\
        F&\longmapsto\beta(F,i_F)= [\clE^{(\mu,\lb;F,i_F)}].
\end{align*}
We recall the definition of $\delta'$ given in \eqref{eq:def
delta'}: $$\delta':\Hom_{S-gr} (G_{\mu,1},\glbp)\too
\Ext^1_S(G_{\mu,1},\glx{1})$$ is defined by
$\delta'(\sigma_i)=\sigma_i^*(\Lambda)$.
 Then, under the
isomorphism \eqref{eq:isomorfismo}, we have
$$
\beta=\rho+\delta':\ker(\widetilde{\psi_{{\lb,1}_*}})\times
\Hom_{S-gr} (G_{\mu,1},\glbp) {\too} \Ext^1_S(G_{\mu,1},\glx{1}).
$$
%

%
\begin{tabular}{|c|}
  \hline
   Injectivity of $\beta$.\\
  \hline
\end{tabular}
First of all we observe that $\widetilde{\delta}$ factors through
$\rho$, i.e.
\begin{equation*}
\widetilde{\delta}=i_*\circ
\rho:\ker(\widetilde{{\psi_{\lb,1}}_*})\on{\rho}{\too}
\Ext^1_S(G_{\mu,1},\glx{1})\on{i_*}{\too}
\ker({\psi_{{\lb,1}_*}}).
\end{equation*}
Indeed 
$$
i_*\circ
\rho(F)=i_*([\clE^{(\mu,\lb;F,i_F)})=[\widetilde{\clE}^{(\mu,\lb,F)}]=\widetilde{\delta}(F).
$$ 
In particular, since $\widetilde{\delta}$ is injective, $\rho$ is
injective, too.

We now prove that $\beta=\rho+\delta'$ is injective, too. By
\eqref{eq:ker
phi bis}, 
$$i_*\circ \delta'=0.$$ Now, if
$(\rho+\delta')(F,\sigma_j)=0$, then $\rho(F)=-\delta'(\sigma_j)$.
So
$$
\widetilde{\delta}(F)=i_*(\rho(F))=i_*(-\delta'(\sigma_j))=i_*(\delta'(\sigma_{-j}))=0.
$$
But $\widetilde{\delta}$ is injective, so $F=1$. Hence
$\delta'(\sigma_j)=0$. But by \eqref{eq:ker phi}, also $\delta'$
is injective. Then $\sigma_j=0$.

\begin{tabular}{|c|}
  \hline
 Calculation of  $\im \beta$.\\
  \hline
\end{tabular}
 We finally prove  $\im(\rho+\delta')=  \ker({\alpha^{(\lb)}_*}\circ
i_*)$. Since $\widetilde{\delta}=i_*\circ \rho$,
$\alpha^{(\lb)}_*\circ\widetilde{\delta}=0$ and
$i_*\circ\delta'=0$ then $$ \alpha^{(\lb)}_*\circ
i_*\circ(\rho+\delta')=0.$$ Then  $\im(\rho+\delta')\In
\ker({\alpha^{(\lb)}_*}\circ i_*)$. On the other hand, if
$[\clE]\in \Ext^1_S(G_{\mu,1},\glx{1})$ is such that
${\alpha^{(\lb)}_*}\circ i_*([\clE])=0$, then, by
\eqref{eq:snake}, there exists $F\in
\ker(\widetilde{\psi_{{\lb,1}_*}})$ such that
$i_*([\clE])=\widetilde{\delta}(F)=i_*(\rho(F))$. Hence, by the
isomorphism \eqref{eq:ker phi bis}, $[\clE]-\rho(F)\in
\im(\delta')$. Therefore $\im(\rho+\delta')=
\ker({\alpha^{(\lb)}_*}\circ i^*)$. Moreover since
$i_*:\Ext^1_S(G_{\mu,1},\glx{1})\too \ker({\psi_{\lb,1}}_*)$ is
surjective then $\im(\alpha^{(\lb)}_*)=\im(\alpha^{(\lb)}_*\circ
i_*)$. Thus we have
 proved, using also \eqref{eq:snake}, that the following sequence
\begin{align*}
0\too \ker(\widetilde{\psi_{{\lb,1}_*}})\times \Hom_{S-gr}
(G_{\mu,1},\glbp) & \on{\rho+\delta'}{\too}
\Ext^1_S(G_{\mu,1},\glx{1})\on{\alpha^{(\lb)}_*\circ i_*}\too\\
\too &\ker \bigg(H^1(S,G_{\mu,1}^\vee)\too
H^1(S_\lb,G_{\mu,1}^\vee)\bigg)
\end{align*}
is exact.
Finally, by definitions, it follows that
$$
\beta(\rad)=\{[\clE^{(\mu,\lb;F,j)}]; (F,j)\in \rad\}.
$$
\end{proof}
\begin{cor} \label{cor:modelli di Z/p^2Z} 
Let $\mu,\lb \in R\setminus\{0\}$. The inverse image of
$\{[\clE_{j,K}]; j\in \Z/p\Z\}$ by the canonical map
$\Ext^1_S(\gmx{1},\glx{1})\too \Ext^1_K(\mu_{p,K},\mu_{p,K})$ is
given by $\{[\clE^{(\mu,\lb;F,j)}]; (F,j)\in \rad\}$. Moreover any
extension $[\clE]\in \ext^1(\gmx{1},\glx{1})$ is of type
$[\clE^{(\mu,\lb;F,j)}]$, up to an extension of $R$.
\end{cor}
\begin{rem}\label{rem:essenzialmente tutti i gruppi di
ordine p2}If $char(R)=0$, any $R$-group scheme  of order $p^2$ is
of type $ \clE^{(\mu,\lb,F,j)}$, possibly after an extension of
$R$. Indeed, up to an extension of $R$, the generic fiber of any
$R$-group scheme is isomorphic to $\clE_{j,K}$. Then the result
follows from \ref{lem:modelli di Z/p^2 Z sono estensioni} and the
previous corollary.
\end{rem}
\begin{proof}

Let us first suppose $\mu,\lb \in \piR$. Let us consider the
following commutative diagram with exact rows
$$
\xymatrix{&\rad\ar[r]&\Ext^1_S(G_{\mu,1},\glx{1})\ar^{\alpha^{(\lb)}_*\circ
i_*}[r]\ar[d]& H^1(S,G_{\mu,1}^\vee)\ar[d]\\
0\ar[r]&\Z/p\Z\simeq [\clE_{i,K}]\ar[r]&
\ext^1_{K}(\mu_{p,K},\mu_{p,K})\ar^{\alpha^{(\lb)}_*\circ i_*}[r]&
 H^1(\Sp(K),\Z/p\Z)\ar[r]& 0}
$$
where the vertical maps are the restrictions to the generic fiber.
The first statement follows, remarking that the second vertical
map is injective by \ref{prop:torsori schemi normali}.

We now prove the second statement. Let $[\clE]\in
\ext^1{(\gmx{1},\glx{1})}$. Suppose that
$\alpha^{(\lb)}_*(i_*[\clE])=[S']$, with $S'\too S$ a
$G_{\mu,1}^\vee$-torsor.  If $S'\too S$ is trivial by the first
exact row of the above diagram we are done. Otherwise $S'_K$ is
integral
and we consider the integral closure  $S''=\Sp(R'')$ of $S$ in
$S'_K$. Since $S''_K\simeq S'_K$, then  $S'_K\times_K S''_K$ is a
trivial $(G_{\mu,1}^\vee)_K$-torsor over $S''_K$. By
\ref{prop:torsori schemi normali} we have that  $S'\times_S S''$
is a trivial $G_{\mu,1}^\vee$-torsor  over
$S''$. 
Let  $\widetilde{R}$ be the  localization of $R''$ at a closed point over $(\pi)$. 
 Then
$\widetilde{R}$ is a noetherian local integrally closed ring of
dimension $1$, i.e. a d.v.r. (see \cite[9.2]{mac}).
Let us call $\widetilde{S}=\Sp(\widetilde{R})$. So, if we make the
base change $f:\widetilde{S}\to S$, then
$\alpha^{(\lb)}_*(i_*([\clE\times_S {\widetilde{S}}]))=0$. Again
by the above diagram, this implies that
 $[\clE\times_S \widetilde{S}]$ is of type
$ [\clE^{(\mu,\lb,F,j)}]. $

The cases $v(\mu)=0$ or $v(\lb)=0$ are similar, even  simpler,
using \ref{eq:ext1(mup,mup)} and \ref{rem: clE sono lo Z/p}.

\end{proof}

\subsection{$\Ext^1_S(G_{\mu,1},\glx{1})$ and the Sekiguchi-Suwa
theory}

 We now give a  description of  the group $\{[\clE^{(\mu,\lb;F,j)}]; (F,j)\in \rad\}$
through the Sekiguchi-Suwa theory. 
We begin with an object which will play a key role in  such a
description.
\begin{lemdefn}
Let $\mu,\lb \in R\setminus\{0\}$. We  define
\begin{align*}
\Phi_{\mu,\lb}:=\bigg\{(a,j)\in (R/\lb
R)^{{\fr}-\mu^{p-1}}\times\Z/p\Z& \text{ such that }
pa-j\mu=\frac{p}{\mu^{p-1}}a^p\in R/\lb^p
R\bigg\}\bigg/\left<(\mu,0)\right>
\end{align*}
if $char(R)=0$, and
$$
\Phi_{\mu,\lb}:=\bigg\{(a,j)\in (R/\lb
R)^{{\fr}-\mu^{p-1}}\times\Z/p\Z \text{ such that } j\mu\equiv 0
\mod \lb^p \bigg\}\bigg/\left<(\mu,0)\right>
$$
if $char(R)=p$. It is a subgroup of
$((R/\lb R)^{{\fr}-\mu^{p-1}}\times\Z/p\Z)/\left<(\mu,0)\right>$. 
\end{lemdefn}
\begin{proof} One should prove that it is a subgroup $((R/\lb
R)^{{\fr}-\mu^{p-1}}\times\Z/p\Z)/\left<(\mu,0)\right>$. In
particular the quotient is well defined. The proof is very easy.
\end{proof}
\begin{rem}\label{rem:v(mu)<v(lb) allora j=0}
If $v(\mu)<v(\lb)$ then $(a,j)\in \Phi_{\mu,\lb}$ implies $j=0$.
If $char(R)=p$ this is trivial. If $char(R)=0$, by
$a^p=\mu^{p-1}a\in R/\lb R$ we obtain
$$pa-j\mu- \frac{p}{\mu^{p-1}}a^p\equiv  -j\mu\mod \lb,$$
which implies $j=0$.
\end{rem}
\begin{prop}\label{prop:rad_p} 
 Let $\mu,\lb \in R\setminus\{0\}$. 
Then 
the map 
\begin{align*}
\Phi_{\mu,\lb}&\too \{[\clE^{(\mu,\lb;F,j)}]; (F,j)\in \rad \} \\
 (a,j)&\longmapsto
[\clE^{(\mu,\lb;E_p(a,\mu,T),j)}]
\end{align*}
is an isomorphism of groups.
\end{prop}
\begin{rem}\label{rem:0 in Phi}
It is clear that if $(0,j)\in \Phi_{\mu,\lb}$, with $j\neq 0$,
then $\mu\equiv 0 \mod \lb^p$.
\end{rem}
%
\begin{proof}
If $v(\lb)=0$ clearly $\Phi_{\mu,\lb}$ is the group $\{0\}\times
    \Z/p\Z$, while if $v(\mu)=0$ and $v(\lb)>0$ it is the trivial
    group. So $\Phi_{\mu,\lb}\simeq \rad \simeq \{[\clE^{(\mu,\lb;F,j)}]; (F,j)\in \rad \}$ in these two cases (see \ref{rem: clE sono lo Z/p}).

 We now suppose $\mu,\lb \in \piR$.
By \ref{teo:ext1(glx,gmx)} we have an isomorphism $\rad
\too\{[\clE^{(\mu,\lb;F,j)}]; (F,j)\in \rad \}$ given by
$(F,j)\mTo [\clE^{(\mu,\lb;F,j)}]$. We now prove that the map
 $\Phi_{\mu,\lb}\too \rad$, given by $(a,j)\too (E_p(a,\mu,T),j)$ is an isomorphism. We distinguish some cases.

\begin{tabular}{|c|}
  \hline
  $char(R)=p$.\\
  \hline
\end{tabular}
By \ref{lem:suriettività mappa tra hom}, \eqref{eq:phi*:W(R)too
W(R)}  and \eqref{eq:ker i_*} it follows that $\rad$ is isomorphic
to
\begin{equation*}
\begin{aligned}\label{eq:rad p} \bigg\{(a,j)\in
(R/\lb R)^{{\fr}-\mu^{p-1}}\times \Z/p\Z|\exists \bbf{b} \in
{\widehat{W}(R/\lb^p R)}^{\fr-[\mu^{p{(p-1)}}]}
\\ \text{ such} \text{ that }
p[a]-j[\mu]=&V(\textit{\textbf{b}})\in {\widehat{W}(R/\lb^p R)}
\bigg\}/\left<(\mu,0)\right>.
\end{aligned}
\end{equation*}
This group is in fact $\Phi_{\mu,\lb}$. Indeed, since $\fr
V=V\fr=p$, then
$j[\mu]=p[a]-V(\textit{\textbf{b}})=V([a^p]-\bbf{b})$ if and only
if $j[\mu]=[j\mu]=0$ and $\textit{\textbf{b}}=[a^p]$.

\begin{tabular}{|c|}
  \hline
  $char(R)=0$ and $p>2$.\\
  \hline
\end{tabular}
 By \ref{lem:suriettività mappa tra hom},  \eqref{eq:phi*:W(R)too W(R)}, \eqref{eq:p a} and
\eqref{eq:ker i_*} it follows that $\rad$ is isomorphic to

\begin{equation*}
\begin{aligned} \bigg\{(a,j)\in
(R/\lb R)^{{\fr}-\mu^{p-1}}\times \Z/p\Z|\exists \bbf{b} \in
{\widehat{W}(R/\lb^p R)}^{\fr-[\mu^{p{(p-1)}}]}
\\ \text{ such} \text{ that }
p[a]-j[\mu]=[\frac{p}{\mu^{p-1}}]&\textit{\textbf{b}}+V(\textit{\textbf{b}})\in
{\widehat{W}(R/\lb^p R)} \bigg\}/\left<(\mu,0)\right>.
\end{aligned}
\end{equation*}
We prove that this group is in fact $\Phi_{\mu,\lb}$.
 Let $a,j$ and
$\textit{\textbf{b}}=(b_0,b_1,\dots)$ be as above. First of all we
observe that if $v(\mu)<v(\lb)$ then,  by \ref{lem:v(mu)>= v(lb)},
$(\clE^{(\mu,\lb;E_p(a,\mu,T),j)})_K=\clE_{j,K}\simeq
\mu_{p,K}\times \mu_{p,K}$. This implies, together to
\ref{rem:v(mu)<v(lb) allora j=0}, that if $v(\mu)<v(\lb)$ we can assume $j=0$. 

Moreover we remark that, by \cite[5.10]{SS4}, we have
\begin{equation*}
p[a]\equiv (pa,a^p,0,\dots)\mod p^2.
\end{equation*}

We first assume  $(p-1)v(\mu)\geq v(\lb)$.  We observe that
$j[\mu]\in \widehat{W}(R/\lb^p R)^{\fr}$. Indeed if $v(\mu)\ge
v(\lb)$ this is clear; while if  $v(\mu)<v(\lb)$ we recall that
$j=0$.
    Since also $p[a], \bbf{b}\in
\widehat{W}(R/\lb^p R)^{\fr}$, then, $V(\bbf{b})=
p[a]-j[\mu]-[\frac{p}{\mu^{p-1}}]\textit{\textbf{b}}\in
\widehat{W}(R/\lb^p R)^{\fr}$; hence by \ref{lem:somma termine per
termine},
\begin{equation}\label{eq:pa-kmu}
p[a]-j[\mu]=(pa-j\mu,a^p,0,0,\dots,0,\dots)\in
\widehat{W}(R/\lb^pR).
\end{equation}
and
$$
[\frac{p}{\mu^{p-1}}]\textit{\textbf{b}}+V(\textit{\textbf{b}})=
(\frac{p}{\mu^{p-1}} b_0, (\frac{p}{\mu^{p-1}})^pb_1+
b_0,\dots,(\frac{p}{\mu^{p-1}})^{p^{i+1}}b_{i+1}+b_i,\dots).
$$
 Comparing the above equations it follows
\begin{align*}
&(\frac{p}{\mu^{p-1}})^{p^{i+1}}b_{i+1}+b_i= 0 \quad\text{ for } i\ge 1\\
&(\frac{p}{\mu^{p-1}})^pb_1+ b_0= a^p     \\
&\frac{p}{\mu^{p-1}}b_0= pa-j\mu.
\end{align*}
Since  there exists $r\ge 0$ such that $b_i=0$ for any $i\ge r$
then $b_i= 0$ if $i\ge 1$,  $b_0=a^p$ and
$pa-j\mu=\frac{p}{\mu^{p-1}}a^p$.

While if $v(\mu)<(p-1) v(\lb)$ then $v(p)-(p-1)v(\mu)>v(\lb)$.
Therefore $(\frac{p}{\mu^{p-1}})^{p^i}=0\in R/\lb^p R$ for $i>0$.
Thus we have
\begin{equation}\label{eq:pmup-1+V(b)}
[\frac{p}{\mu^{p-1}}]\textit{\textbf{b}}+V(\textit{\textbf{b}})=(\frac{p}{\mu^{p-1}}b_0,0,\dots)+
V(\textit{\textbf{b}})=(\frac{p}{\mu^{p-1}}b_0,b_0,b_1,\dots).
\end{equation}
Since $j=0$  we have
$$
(pa, a^p,0,\dots)=(\frac{p}{\mu^{p-1}}b_0,b_0,b_1,\dots).
$$
This gives $b_i=0$ if $i\ge 0$, $b_0=a^p$ and
$$
pa-j\mu =pa=\frac{p}{\mu^{p-1}}a^p
$$
as required.

\begin{tabular}{|c|}
  \hline
  $char(R)=0$ and  $p=2$ \\
  \hline
\end{tabular}
Let us consider the isomorphism
$$
(\xi^0_{R/\lb R})_p\times \id: (R/\lb R)^{\fr-[\mu]}\times
\Z/2\Z\too \Hom_{S_\lb-gr}(G_{\mu,1},\gmSl)\times \Z/2\Z.
$$
We have just to prove that $((\xi^0_{R/\lb R})_p\times
\id)^{-1}(\rad)=\Phi_{\mu,\lb}$.

Let $(F,j)\in \Hom_{S_\lb-gr}(G_{\mu,1},\gmSl)\times \Z/2\Z$. Then
$F(T)=1+a T$. Moreover the class of $ (F,j) \mod \left< (1+\mu
T,0)\right >$ belongs to $\rad $ if and only if
\begin{equation}\label{eq:(F,j) in rad p=2}
(1+a T)^2 =(1+ \mu T)^j=(1+j \mu T) \in
\Hom_{S_\lb-gr}(G_{\mu,1},\gmSl)
\end{equation}
Recall that $T^2 =-\frac{2}{\mu} T\in R[G_{\mu,1}]$. Therefore
\eqref{eq:(F,j) in rad p=2} is satisfied if and only if
$$
2a  -\frac{2}{\mu}a^2= j \mu \in R/\lb^2 R.
$$
This is equivalent to say the class of $(a,j)\mod
\left<(\mu,0)\right> $ belongs to  $\Phi_{\mu,\lb}$.

\end{proof}


We now  find explicitly all the solutions $(a,j)\in (R/\lb
R)^{{\fr}-\mu^{p-1}}\times \Z/p\Z$ of the equation $pa-j\mu=a^p\in
R/\lb^p R$. By \ref{prop:rad_p} this means  finding explicitly all
the extensions of type $\clE^{(\mu,\lb;F,j)}$. Let us consider the
restriction map
$$
 r: \{[\clE^{(\mu,\lb;F,j)}]; (F,j)\in \rad \} \too H_0^2(\mu_{p,K},\mu_{p,K})\simeq \Z/p\Z.
$$
We remark that it coincides with the projection on the second
component
\begin{align*}
p_2:\Phi_{\mu,\lb}\too \Z/p\Z.
\end{align*}

 The proof of the following lemma is trivial.
\begin{lem}\label{lem:modelli esistono ss p_2 surj}
 There is an extension of
$\gmx{1}$ by $\glx{1}$ which is a model
 of $\mu_{p^2,K}$ if and only if $p_2$ is surjective.
\end{lem}

We now describe  the kernel of the above map.
\begin{lem}\label{lem:ker p2}
Under the assumptions of \ref{prop:rad_p} we have
\begin{align*}
\ker p_2= \bigg\{(a,0)\in R/\lb R\times&\Z/p\Z \text{ s. t., for
any
  lifting }\widetilde{a}\in R,\\ &
v(\widetilde{a}^p-\mu^{p-1}\widetilde{a})\ge\max\{pv(\lb)+(p-1)v(\mu)-v(p),v(\lb)\}\bigg\}\bigg/\left<{(\mu,0)}\right>
\end{align*}
if $char(R)=0$, and
$$
\ker p_2=\bigg( ( R/\lb
R)^{\fr-\mu^{p-1}}\times\{0\}\bigg)\bigg/\left<{(\mu,0)}\right>
$$
if $char(R)=p$.
\end{lem}
\begin{proof}
The case $char(R)=p$ is trivial. Let us now suppose $char(R)=0$.
Let $(a,0)\in \ker p_2\cap R/\lb R\times \Z/p\Z $. By definition
we have that
$$
pa=\frac{p}{\mu^{p-1}}a^p\in R/\lb^p R \quad \text{ and } \quad
a^p-\mu^{p-1}a=0\in R/\lb R.
$$
Let $\widetilde{a}\in R$ be a lifting of $a$. Since
$a^p=\mu^{p-1}a\in R/\lb R$ then
$\widetilde{a}^p-\mu^{p-1}\widetilde{a}=b$ with $v(b)\ge v(\lb)$. 
Hence we have $pa=\frac{p}{\mu^{p-1}}a^p\in R/\lb^p R$ if and only
if $\frac{p}{\mu^{p-1}}b=0 \in R/\lb^p R$. This happens if and
only if $v(\widetilde{a}^p-\mu^{p-1}\widetilde{a})\ge
pv(\lb)+(p-1)v(\mu)-v(p)$.

%
\end{proof}
\begin{ex}\label{ex:ker p_2}
Let $char(R)=0$. If $v(\mu)\ge v(\lb)$, which is the case of
models of $\mu_{p^2,K}$ by \ref{lem:v(mu)>= v(lb)}, then
\begin{align*}\ker p_2= \bigg\{(a,0)\in R/\lb R\times &\Z/p\Z
\text{ s. t., for any
  lifting }\widetilde{a}\in R,\\ &
pv(\widetilde{a})\ge\max\{pv(\lb)+(p-1)v(\mu)-v(p),v(\lb)\}\bigg\}
\end{align*}
and  it is easy to show that
 $p_2$ is injective if and only if $v(\lb)\le 1$ or
$v(p)-(p-1)v(\mu)<p$.
\end{ex}
 We now compute $\Phi_{\mu,\lb}$.
\begin{prop}\label{cor:p2 surjective}
Let $\mu,\lb\in \pi R\setminus\{0\}$. In the case $char(R)=0$ we
will suppose that $R$ contains a distinguished primitive $p^2$-th
root of unity $\zeta_2$ (and
$(p-1)v(\mu),(p-1)v(\lb)\le v(p)$ as usual).

Let $char(R)=0$.
\begin{itemize}
\item[(i)]  If $v(\mu)<v(\lb)$, or   $v(\lb)\le
v(\mu)<pv(\lb)$ and
$pv(\mu)-v(\lb)< v(p)$ then $p_2$ is trivial. 
\item[(ii)]If    $v(\lb)\le v(\mu)<pv(\lb)$ and
$pv(\mu)-v(\lb)\ge v(p)$ then $p_2$ is surjective and $
\Phi_{\mu,\lb}$ is isomorphic to  the group
$$
\{(j \eta \frac{\mu}{\lb_{(1)}}+ \alpha,j); (\alpha,0)\in
\ker(p_2) \text{ and } j\in \Z/p\Z\}
$$
For the definitions of $\eta$  see \ref{ex:Z/p^2Z SS}.
\item[(iii)] If $v(\mu)\ge pv(\lb)$, then $p_2$ is surjective and
$ \Phi_{\mu,\lb}$ is isomorphic to the group $$ \{(\alpha,j);
(\alpha,0)\in \ker(p_2) \text{ and } j\in \Z/p\Z\}.
$$
\end{itemize}
Let $char(R)=p$.
\begin{itemize}
\item[(iv)]$p_2$ is surjective if and only if $v(\mu)\ge pv(\lb)$.
If $p_2$ is surjective then $ \Phi_{\mu,\lb}$ is isomorphic to the
group
$$ \{(\alpha,j); (\alpha,0)\in \ker(p_2) \text{ and } j\in
\Z/p\Z\}
$$
\end{itemize}
\end{prop}
\begin{rem}\label{rem:valutazione di elementi di Phi1}Let us suppose $v(\lb)\le v(\mu)<pv(\lb)$. Let $(b,j)\in \Phi_{\mu,\lb}$ with $j\neq 0$. By \ref{rem:0 in Phi}, then $b\neq 0$. Let $\widetilde{b}\in
R$ be any of its lifting. Then
$v(\widetilde{b})=v(\eta\frac{\mu}{\lb_{(1)}})=v(\mu)-\frac{v(p)}{p}$.
Indeed, by the proposition, we have
$\widetilde{b}=j\eta\frac{\mu}{\lb_{(1)}}+\alpha$ for some $\alpha
\in R/\lb R$ with
$v(\widetilde{\alpha})>v(\eta\frac{\mu}{\lb_{(1)}})=v(\mu)-\frac{v(p)}{p}$,
where $\widetilde{\alpha}\in R$ is any lifting of $\alpha$.
\end{rem}
\begin{rem}We need the hypothesis $\zeta_{2}\in R$   just in the proof of (ii). If we  knew  an explicit solution $a$ (if there exists) of the
equation $pa-j \mu=\frac{p}{\mu^{p-1}}a \in R/\lb^p R$, also in
the case $(p-1)v(\mu)=(p-1)v(\lb)=v(p)$ but $\zeta_{2}\not \in
R$,    then a statement as above holds. This solution would play
the role of $\eta \frac{\mu}{\lb_{(1)}}$.
\end{rem}
\begin{proof}
\begin{itemize}
\item[i)] If $v(\mu)<v(\lb)$ the result follows from
\ref{lem:v(mu)>= v(lb)}. Let us now suppose $v(\lb)\le
v(\mu)<pv(\lb)$ and $pv(\mu)-v(\lb)< v(p)$. Since the target of
$p_2$ is $\Z/p\Z$ then $p_2$ is either surjective or trivial. Let
us suppose it is surjective. This is equivalent to saying that
\begin{equation}\label{eq:soluzione pa-jmu...}
pa-j\mu=\frac{p}{\mu^{p-1}}a^p\in R/\lb^p R
\end{equation}
has a solution $a\in (R/\lb R)^{\fr}$ with $j\neq 0$. Since
$v(\mu)<pv(\lb)$ then $a\neq 0$. Let $\widetilde{a}\in R$ be a
lifting of $a$. Since $v(\mu)<v(p)$ and
$v(\frac{p}{\mu^{p-1}}\widetilde{a}^p) < v(p \widetilde{a})$, then
\begin{equation}\label{eq:soluzione pa=... segue che...}
v(\mu)=v(p)-(p-1)v(\mu)+pv(\widetilde{a}),
\end{equation}
 From $a\in (R/\lb
R)^{{\fr}}$ it follows $pv(\widetilde{a})\ge v(\lb)$. Hence, by
\eqref{eq:soluzione pa=... segue che...}, $pv(\mu)- v(\lb)\ge
v(p)$, against the hypothesis.

\item[ii)]
We know by \ref{ex:Z/p^2Z SS} and \ref{prop:rad_p} that
$$
p\eta-\lb_{(1)}=\frac{p}{\lb_{(1)}^{p-1}}\eta^p\in R/ \lb_{(1)}^p
R.
$$
We recall that $v(\eta)=v(\lb_{(2)})$.  Since $p
v(\lb_{(1)})-v(\lb_{(1)})+v(\mu)\ge pv(\mu)\ge pv(\lb)$, if we
divide the above equation by $\frac{\lb_{(1)}}{\mu}$, we obtain
\begin{equation*}
p\eta \frac{\mu}{\lb_{(1)}}-\mu=
\frac{p}{\mu^{(p-1)}}(\frac{\mu}{\lb_{(1)}}\eta)^p\in R/\lb^p R.
\end{equation*}
We remark that $\eta \frac{\mu}{\lb_{(1)}}\in (R/\lb R)^{{\fr}}$,
since, by hypothesis, $v(\big(\eta
\frac{\mu}{\lb_{(1)}}\big)^p)=pv(\mu)-v(p)\ge v(\lb)$. Clearly
$j\eta \frac{\mu}{\lb_{(1)}}$ is a solution of \eqref{eq:soluzione
pa-jmu...} for any $j\in \Z/p\Z$.

Therefore  $p_2$ is surjective and  $\Phi_{\mu,\lb}$ is isomorphic
to the group
$$
\{(j \eta \frac{\mu}{\lb_{(1)}}+ \alpha,j); (\alpha,0)\in
\ker(p_2) \text{ and } j\in\ \Z/p\Z\}
$$
\item[iii)]
If $v(\mu)\ge pv(\lb)$ then we have that $\mu=0\in R/\lb^p R$.  We
remark that $(0,j)\in \Phi_{\mu,\lb}$ for any $j\in \Z/p\Z$. This
implies that $p_2$ is surjective and, moreover, that
$(\alpha,j)\in (R/\lb R\times \Z/p\Z) \cap \Phi_{\mu,\lb}$ if and
only if $(\alpha,0)\in \ker(p_2)$.

\item[(iv)] 
The result immediately follows from the explicit description of
the map $p_2$ in the case $char(R)=p$.
\end{itemize}
\end{proof}
\begin{ex} Assume $char(R)=0$, $\zeta_{2}\in R$ and $v(\mu)=v(\lb)>0$. Then $p_2$ is surjective 
if and only if  $v(\mu)=v(\lb)=v(\lb_{(1)})$.
\end{ex}
\begin{ex}\label{ex:Ext 1 con mu=lb_(1)}
Let us suppose $char(R)=0$, $\zeta_{2}\in R$ and  
$v(\mu)=v(\lb_{(1)})\ge v(\lb)$. We have  $G_{\mu,1}\simeq
\Z/p\Z$. Without loss of generality we can suppose
$\mu=\lb_{(1)}$. Then $p_2$ is an isomorphism. Indeed in this case
$\ker(p_2)=0$ by \ref{ex:ker p_2} and it is surjective by
\ref{cor:p2 surjective}(ii). However it is possible to prove this
fact in a more direct way. We have the following commutative
diagram of exact rows
$$
\xymatrix{0\ar[r] & H_0^2(\Z/p\Z,\glx{1})\ar[d]\ar[r]&
\Ext^1_S(\Z/p\Z,\glx{1})\ar[r]\ar[d]&H^1(S,\glb)\ar[r]\ar[d]&0\\
 0\ar[r] & H_0^2(\Z/p\Z,\Z/p\Z)\ar[r]&
\Ext^1_K(\Z/p\Z,\Z/p\Z)\ar[r]&H^1(K,\glb)\ar[r]&0}
$$
where the exact rows come from \cite[III \S 6, 2.5]{DG}. Moreover
since $\Z/p\Z$ is constant and cyclic, using\cite[III \S 6,
4.2]{DG} and \cite[VIII \S4]{ser}   we have
$$
H_0^2(\Z/p\Z,\glx{1})\simeq G_{\lb,1}(S), \qquad \text{ and
}\qquad H_0^2(\Z/p\Z,(\glx{1})_K)\simeq G_{\lb,1}(K).
$$
Therefore the first vertical map is an isomorphism. It is in fact
$p_2$. This finally implies that there is an
 isomorphism
$$ H_0^2(\Z/p\Z,\glx{1})\simeq \{[\clE^{(\lb_{(1)},\lb;F_j,j)}] ;
j=0,\dots,p-1\},$$ where $F_j(T)=\sum_{k=0}^{p-1}\frac{(j
\eta)^k}{k!} T^k$.
\end{ex}

\subsection{Classification of models of $\mu_{p^2,K}$}
By the previous paragraphs we have a classification of extensions
of $G_{\mu,1}$ by $G_{\lb,1}$ whose generic fibre is isomorphic,
as group scheme, to $\mu_{p^2,K}$. But this classification is, a
priori, too fine for our tasks. We want here to forget the
structure of extension. We are only interested in the group scheme
structure. We observe that it could happen that two non isomorphic
extensions are isomorphic as group schemes. We here study when it
happens.

%
%
We now recall that by \ref{lem:modelli di Z/p^2 Z sono
estensioni},  \ref{cor:modelli di Z/p^2Z} and \ref{prop:rad_p} any
model of $\mu_{p^2,K}$ is of the form
$\clE^{(\mu,\lb;\tilde{F},j)}$ such that $j\neq 0$,
$\tilde{F}(T)\equiv \sum_{i=0}^{p-1}\frac{a^i}{i!}T^i\mod \lb$
with $(a,j)\in \Phi_{\mu,\lb}$. Moreover if $char(R)=0$ then
$v(p)\ge (p-1)v(\mu)\ge (p-1)v(\lb)$ (see  \ref{lem:v(mu)>= v(lb)}
for $v(\mu)\ge v(\lb)$), while if $char(R)=p$ then $v(\mu)\ge
pv(\lb)$ (see \ref{cor:p2 surjective}(iv)).  In these cases we
have
$$
\Phi_{\mu,\lb}=\bigg\{(a,j)\in (R/\lb R)^{{\fr}}\times\Z/p\Z
\text{ such that } pa-j\mu=\frac{p}{\mu^{p-1}}a^p\in R/\lb^p
R\bigg\}
$$
if $char(R)=0$  and
$$
\Phi_{\mu,\lb}=\bigg\{(a,j)\in (R/\lb R)^{{\fr}}\times\Z/p\Z
\text{ such that } j\mu\equiv 0\mod \lb^p \bigg\}
$$
if $char(R)=p$.
For $i=1,2$ let us consider $\clE^{(\mu_i,\lb_i;F_i,j_i)}$, models
of $\mu_{p^2,K}$. First of all we remark that there is an
injection
$$
r_K:\Hom_{S-gr}(\clE^{(\mu_1,\lb_{1};F_1,j_1)},
\clE^{(\mu_2,\lb_{2};F_2,j_2)})\too
\Hom_{K-gr}(\clE_{j_1,K},\clE_{j_2,K})
$$
given by
$$
f\longmapsto (\alpha^{(\mu_2,{\lb_2},\tilde{G})})_K\circ f_K \circ
{(\alpha^{(\mu_1,\lb_1,\tilde{F})})}^{-1}_K.
$$
  We recall that
$$
\Hom_{S-gr}({\clE_{j_1}},{\clE_{j_2}})\simeq
\Hom_{K-gr}(\clE_{j_1,K},\clE_{j_2,K})
$$
and the  elements are the morphisms
$$
\psi_{r,s}:{\clE_{j_1}}\too \clE_{j_2},
$$
which, on the level of Hopf algebras, are given by
\begin{equation}\label{eq:morfismi tra cle j}
\begin{aligned}
T_1&\longmapsto T_1^{\frac{rj_1}{j_2}}\\
T_2&\longmapsto T_1^sT_2^r,
\end{aligned}
\end{equation}
for some $r,s=0,\dots,p-1$.  Moreover the map
\begin{align*}
\Hom_{S-gr}({\clE_{j_1}},{\clE_{j_2}})&\too \Z/p^2\Z\\
\psi_{r,s}&\longmapsto r+\frac{p}{j_1}s
\end{align*}
is an isomorphism. So $\Hom_{S-gr}(\clE^{(\mu_1,\lb_{1};F_1,j_1)},
\clE^{(\mu_2,\lb_{2};F_2,j_2)})$ is a subgroup of $\Z/p^2\Z$
through the map $r_K$. We remark that the unique nontrivial
subgroup of $\Hom({\clE_{j_1}},{\clE_{j_2}})$ is
$\{\psi_{0,s};s=0,\dots,p-1\}$. Finally we have that any morphism
$\clE^{(\mu_1,\lb_{1},F_1,j_1)}\too
\clE^{(\mu_2,\lb_{2},F_2,j_2)})$ is given by
\begin{equation}\label{eq:def f}
\begin{aligned}
T_1&\too \frac{(1+\mu_1 T_1)^{\frac{rj_1}{j_2}}-1}{{\mu_2}} \\
T_2&\too \frac{(F_1(T_1)+\lb_{1} T_2)^{r}(1+\mu_1
T_1)^s-F_2(\frac{(1+\mu_1
T_1)^{\frac{rj_1}{j_2}}-1}{{\mu_2}})}{{\lb_{2}}},
\end{aligned}
\end{equation}
for some $r,s\in \Z/p\Z$. With abuse of notation we call it
$\psi_{r,s}$. We remark that the morphisms
$\psi_{r,s}:\clE^{(\mu_1,\lb_1;F_1,j_1)}\too
\clE^{(\mu_2,\lb_2;F_2,j_2)}$ which are model maps correspond, by
\eqref{eq:morfismi tra cle j}, to $r\neq 0$. In such a case
$\psi_{r,s}$ is a morphism of extensions, i.e. there exist
morphisms $\psi_1:G_{\lb_1,1}\too G_{\lb_2,1}$ and
$\psi_2:G_{\mu_1,1}\too G_{\mu_2,1}$ such that
\begin{equation}\label{eq:commutative diagram}
\begin{aligned}
\xymatrix@1{0\ar[r]&{G_{\lb_1,1}}\ar[r]^{}\ar[d]^{\psi_1}&{\clE^{(\mu_1,\lb_1;F_1,j_1)}}\ar^{\psi_{r,s}}[d]\ar[r]&{G_{\mu_1,1}}\ar[r]\ar[d]^{\psi_2}&0\\
                     0\ar[r]&{G_{\lb_2,1}}\ar[r]^{}&\clE^{(\mu_2,\lb_2;F_2,j_2)}\ar[r]^{}&{G_{\mu_2,1}}\ar[r]&0}
\end{aligned}
\end{equation}
commutes. More precisely $\psi_1$ is given by $T\mapsto
\frac{(1+\lb_1 T)^{r}-1}{{\lb_2}}$ and $\psi_2$ by $T\mapsto
\frac{(1+\mu T)^{\frac{rj_1}{j_2}}-1}{{\mu_2}}$. 

\begin{prop}
 For $i=1,2$, if $F_i(T)=\sum_{k=0}^{p-1}\frac{a_i^k}{k!}T^i$ and
${\scE}_i=\clE^{(\mu_i,\lb_{i};F_i,j_i)}$ are models of
$\mu_{p^2,K}$ we have
$$
\Hom_{S-gr}(\scE_1,\scE_2)=
\left\{%
\begin{array}{ll}
    0, & \hbox{\text{if $v(\mu_1)<v(\lb_2)$;}} \\
    \{\psi_{r,s}\}\simeq \Z/p^2\Z, & \hbox{\text{if $v(\mu_2)\le v(\mu_1)$, $v(\lb_2)\le v(\lb_1)$}}\\
    &  \text{and $a_1\equiv \frac{j_1}{j_2}\frac{\mu_1}{\mu_2}a_2\mod{\lb_2}$;}\\
\{\psi_{0,s}\}\simeq \Z/p\Z, & \hbox{\text{otherwise}}.
\end{array}%
\right.
$$

\end{prop}
\begin{proof}
It is immediate to see that $\psi_{0,s}\in
\Hom_{S-gr}(\scE_1,\scE_2)$, with $s\neq 0$,  if and only if
$v(\mu_1)\ge v(\lb_2)$. We now see conditions for the existence of
$\psi_{r,s}$ with $r\neq 0$.
 If it exists, in particular, we have two  morphisms
 $G_{\mu_1,1}\too G_{\mu_2,1}$ and $G_{\lb_1,1}\too G_{\lb_2,1}$.
This implies $v(\mu_1)\ge v(\mu_2)$ and $v(\lb_1)\ge v(\lb_2)$.
Moreover we have that
$$
F_1(T_1)^r(1+\mu_1 T_1)^s=F_2\bigg(\frac{(1+\mu_1
T_1)^{\frac{rj_1}{j_2}}-1}{{\mu_2}}\bigg)\in
\Hom_{S_{\lb_{2}}-gr}({G_{\mu_1,1}},{\gmSl}).
$$
Since $v(\mu_1)\ge v(\mu_2)\ge v(\lb_2)$,  we  have
\begin{equation*}
F_1(T_1)^r=F_2\bigg(\frac{(1+\mu_1
T_1)^{\frac{rj_1}{j_2}}-1}{{\mu_2}}\bigg)\in
\Hom_{S_{\lb_{2}}-gr}({G_{\mu_1,1}},{\mathbb{G}_{m,S_{\lb_2}}}).
\end{equation*}
If we define the morphism of groups
\begin{align*}
[\frac{\mu_1}{\mu_2}]^*:\Hom_{S_{\lb_{2}}-gr}({G_{\mu_2,1}},\mathbb{G}_{m,S_{\lb_{2}}})&\too
\Hom_{S_{\lb_{2}}-gr}({G_{\mu_1,1}},\mathbb{G}_{m,S_{\lb_2}})\\
F(T_1)&\longmapsto F(\frac{\mu_1}{\mu_2}T_1)
\end{align*}
then
\begin{align*}
F_2\bigg(\frac{(1+\mu_1
T_1)^{\frac{rj_1}{j_2}}-1}{\mu_2}\bigg)&=[\frac{\mu_1}{\mu_2}]^*\bigg(F_2\bigg(\frac{(1+\mu_1
T_1)^{\frac{rj_1}{j_2}}-1}{\mu_1}\bigg)\bigg)\\&=[\frac{\mu_1}{\mu_2}]^*(F_2(T_1))^{\frac{rj_1}{j_2}}\\
&=F_2(\frac{\mu_1}{\mu_2}(T_1))^{\frac{rj_1}{j_2}}.
\end{align*}
Therefore we have
\begin{equation*}
F_1(T_1)^r=(F_2(\frac{\mu_1}{\mu_2}T_1))^{\frac{rj_1}{j_2}}\in
\Hom_{S_{\lb_{2}}-gr}({G_{\mu_1,1}},{\mathbb{G}_{m,S_{\lb_2}}}).
\end{equation*}
Any element of $\Hom_{{S_\lb}-gr}({G_{\mu_1,1}},{\gmSl}) $ has
order $p$. Let $t$ be an inverse for $r$  modulo $p$. Then raising
the equality to the $t$-{th} power we obtain
$$
F_1(T_1)=(F_2(\frac{\mu_1}{\mu_2}T_1))^{\frac{j_1}{j_2}}\in\Hom_{S_{\lb_2}-gr}({G_{\mu_1,1}},{\mathbb{G}_{m,S_{\lb_2}}}).
$$
This means
$$
a_1\equiv {\frac{j_1}{j_2}}\frac{\mu_1}{\mu_2}a_2\mod \lb_2.
$$

  Conversely it is  clear
that,  if  $v(\mu_1)\ge v(\mu_2)$, $v(\lb_{1})\ge v({\lb_{2}})$
and
$$
F_1(T_1)=(F_2(\frac{\mu_1}{\mu_2}T_1))^{\frac{j_1}{j_2}}\in
\Hom_{S_{\lb_{2}}-gr}({G_{\mu_1,1}},{\gmSl}),
$$
then \eqref{eq:def f} defines a morphism of group schemes. 

 \end{proof}
We have the following result which gives a criterion to determine
the class of isomorphism, as a group scheme, of an extension of
type $\clE^{(\mu,\lb;F,j)}.$
\begin{cor}\label{cor:iso tra modelli}
For $i=1,2$, let $F_i(T)=\sum_{k=0}^{p-1}\frac{a_i^k}{k!}T^k$ and
let $\scE_i=\clE^{(\mu_i,\lb_{i};\tilde{F}_i,j_i)}$ be models of
$\mu_{p^2,K}$, with $\tilde{F}_i$ liftings of $F_i$. Then they are
isomorphic if and only if $v(\mu_1)=v(\mu_2)$, $v(\lb_1)=v(\lb_2)$
and $a_1\equiv
    \frac{j_1}{j_2}\frac{\mu_1}{\mu_2}a_2\mod{\lb_2}$. Moreover if
    it happens then any model map between them is an isomorphism.
\end{cor}
\begin{rem}The last sentence is in fact true more in general: any
model map between isomorphic finite and flat commutative $R$-group
schemes is in fact an isomorphism (see \cite[Cor. 3]{far}).
\end{rem}
\begin{proof}
By the proposition we have that  a model map
$\psi_{r,s}:\clE^{(\mu_1,\lb_1,F_1,j_1)}\too
\clE^{(\mu_2,\lb_2,F_2,j_2)}$  exists if and only if $v(\mu_1)\ge
v(\mu_2)$, $v(\lb_1)\ge v(\lb_2)$ and $a_1\equiv
\frac{j_1}{j_2}\frac{\mu_1}{\mu_2}a_2\mod{\lb_2}$. It is a
morphism of extensions as remarked before the proposition.
    Let us consider the commutative diagram \eqref{eq:commutative diagram}. Then $\psi_{r,s}$ is an isomorphism if and only if $\psi_i$ is an isomorphism for $i=1,2$.
   This is equivalent to requiring $v(\mu_1)=v(\mu_2)$ and $v(\lb_{1})=v(\lb_{2})$.
   This also proves the last
 assertion. 
\end{proof}

We conclude the section with the complete classification of
$\mu_{p^2,K}$-models. The following theorem summarizes the above
results.

 \begin{thm}\label{cor:modelli di Z/p^2 Z sono cosi'2}
 Let $G$ be a finite and flat $R$-group
scheme such that $G_K\simeq \mu_{p^2,K}$. Then $G\simeq
\clE^{(\pi^{m},\pi^{n};\tilde{F},1)}$ for some $m,n\ge 0$,
$\tilde{F}(T)$ a lifting of
$F(T)=\sum_{k=0}^{p-1}\frac{a^k}{k!}T^k$ with $(a,1)\in
\Phi_{\pi^m,\pi^n}$. If $char(R)=0$ then $m\ge n$ and $(p-1)m\le
v(p)$, while if $char(R)=p$ then $m\ge pn$. Moreover $m,n$ and
$a\in R/\pi^{n} R$ are unique.
\end{thm}
 \begin{proof} By \ref{lem:modelli di Z/p^2 Z sono estensioni}, \ref{cor:modelli
di Z/p^2Z} and      \ref{prop:rad_p} 
 any model of $\mu_{p^2,K}$ is of type $\clE^{(\pi^m,\pi^n;\tilde{G},j)}$ for some
 $m,n\ge 0$, and $\tilde{G}(T)$ a lifting of $G(T)=\sum_{k=0}^{p-1}\frac{b^k}{k!}T^k$ with $(b,j)\in\Phi_{\pi^m,\pi^n}$ and $j\neq 0$.
Moreover,  by \ref{lem:v(mu)>= v(lb)} and by definition of group
schemes $G_{\pi^m,1}$, if $char(R)=0$ then $m\ge n$ and $(p-1)m\le
v(p)$. While if $char(R)=p$, by \ref{lem:modelli esistono ss p_2
surj} and \ref{cor:p2 surjective}(iv),  then $m\ge pn$.

Now  we  prove $(\frac{b}{j},1)\in \Phi_{\pi^m,\pi^n}$. If
$char(R)=p$ this is trivial. We now assume $char(R)=0$.
Since $(b,j)\in \Phi_{\pi^m,\pi^n}$ 
then  $b\in (R/\lb R)^{{\fr}}$ and
\begin{equation}\label{eq:equazione per bbf a}
pb-j\pi^m=\frac{p}{\pi^{m(p-1)}}b^p\mod\pi^{np}.
\end{equation}
Clearly $\frac{b}{j}\in (R/\lb R)^{\fr}$. Moreover, multiplying
\eqref{eq:equazione per bbf a} by $\frac{1}{j}$, we  have
$$
p\frac{b }{j}-\pi^m\equiv \frac{p}{\pi^{m(p-1)}}\bigg(\frac{b}{
j}\bigg)^p\mod {\pi^{np}}.
$$

Let $a:=\frac{b}{j}$, $F(T)=\sum_{k=0}^{p-1}\frac{a^k}{k!}T^k$ and
$\tilde{F}(T)$ a lifting of $F(T)$. Then by \ref{cor:iso tra
modelli} we can conclude that
$$
\clE^{(\pi^m,\pi^n;\tilde{G},j
)}\simeq\clE^{({\pi^m},\pi^n;\tilde{F},1)}.
$$
Moreover again by \ref{cor:iso tra modelli}, it follows that
$(a_i,1)\in \Phi_{m_i,n_i}$, for $i=1,2$, correspond to two
isomorphic models  of  $\mu_{p^2,K}$
if and only if $m_1=m_2$, $n_1=n_2$ and $a_1=a_2\in R/\pi^{n_1} R$.  
\end{proof}

\section{Reduction on the special fiber}
In the following we study the special fibers of the extension
classes of type $[\clE^{(\mu,\lb;F,j)}]$ with $(F,j)\in \rad$. In
this section, if $char(R)=0$ we suppose  that $R$ contains a
distinguished primitive $p^2$-th root of unity $\zeta_2$.
We remark that the special fibers could be in one of the following
$\Ext^1_k$
\begin{align*}
\Ext^1_k(\mu_{p,k},\mu_{p,k}) \qquad& \Ext^1_k(\alpha_{p,k},\mu_{p,k})\\
\Ext^1_k(\mu_{p,k},\alpha_{p,k}) \qquad &
\Ext^1_k(\alpha_{p,k},\alpha_{p,k})
\end{align*}
and moreover, if $char(R)=0$,
$$
\begin{array}{ccc}
\Ext^1_k(\mu_{p,k},\Z/p\Z)  & \Ext^1_k(\Z/p\Z,\mu_{p,k})& \Ext^1_k(\Z/p\Z,\Z/p\Z)\\
\Ext^1_k(\Z/p\Z,\alpha_{p,k}) & \Ext^1_k(\alpha_{p,k},\Z/p\Z).\\
\end{array}
$$
We recall that we consider only commutative extension.
 We study separately the different
cases which can occur.


%
%
\subsection{Case $\mathbf{v(\mu)=v(\lb)=0}$}
By \ref{eq:ext1(mup,mup)}(i), \ref{rem:Ext(mup,mup) ogni base} and
\ref{rem: clE sono lo Z/p} we have
$$
\xymatrix{0\ar[r] & \{ [\clE_{j,S}]; j\in \Z/p\Z\}\ar[d]\ar[r]&
\Ext^1_S(\mu_{p,S},\mu_{p,S})\ar[r]\ar[d]&H^1(S,\Z/p\Z)\ar[r]\ar[d]&0\\
 0\ar[r] & \{[\clE_{j,k}]; j\in \Z/p\Z\}\ar[r]&
\Ext^1_k(\mu_{p,k}, \mu_{p,k})\ar[r]&H^1(k,\Z/p\Z)\ar[r]&0}
$$
where the vertical maps are the restriction maps. Clearly the
first vertical map is an isomorphism.

%
\subsection{Case $\mathbf{ v(\mu)>v(\lb)=0}$}
In such a case we have $$\{[\clE^{(\mu,\lb,;F,j)}]; (F,j)\in
\rad\}=\{[\clE^{(\mu,\lb;1,j)}];j\in \Z/p\Z\}.$$
 It is immediate to see  that any  extension $[\clE^{(\mu,\lb;1,j)}]$ is trivial
on the special fiber.
\subsection{Case $\mathbf{ v(\lb)>v(\mu)=0}$}
In this case  $\{[\clE^{(\mu,\lb,;F,j)}]; (F,j)\in \rad\}$ is
trivial (see \ref{rem: rad=0}).

\subsection{Case $\mathbf{v(\mu), v(\lb)>0}$ and, if char(R)=0, $\mathbf{ v(\mu),v(\lb)< v(\lb_{(1)})}$}

 Then ${(G_{\mu,1})}_k\simeq (G_{\lb,1})_k\simeq
{\ap}_{,k}$. First, we recall some results about extensions, over
 $k$, with quotient $\alpha_{p,k}$. See
\cite[II \S 3 n°4, III \S 6 n°7]{DG} for a reference.

\begin{thm}\label{teo:ext(ap,ap)}
The exact sequence $$ 0\too \alpha_{p,k}\to \Gak\on{{\fr}}\to
\Gak\to 0$$ induces the following split exact sequence
$$
0\too \Hom_{k-gr} (\alpha_{p,k},\Gak)\too \Ext^1_k(\apk,\apk)\too
\Ext^1_k(\apk,\Gak)\too 0.
$$
\end{thm}
It is also known that
 $$
\Ext^1_k(\Gak,\Gak)\simeq H_s^2(\Gak,\Gak)\too
H_0^2(\apk,\Gak)\simeq \Ext^1_k(\apk,\Gak).
 $$
 is surjective, where $H_s^2(*,*)$ denotes the Hochschild group of simmetric cocycles.
 Since $ H_s^2(\Gak,\Gak)$ is
 freely generated as a right $k$-module by $2$-cocycles
 $C_i=\frac{U^{p^i}+V^{p^i}-(U+V)^{p^i}}{p}$,  for all $i\in \N\setminus\{0\}$, it
 follows that
 $H_0^2(\apk,\Gak)$ is freely generated as
right $k$-module by the class of the  cocycle
$C_1=\frac{U^p+V^p-(U+V)^p}{p}$. So $\Ext^1_k(\apk,\Gak)\simeq k$.

 Moreover it is easy to see that $\Hom_k (\alpha_{p,k},\Gak)\simeq k$.
 The morphisms are given by $T\mapsto a T$ with $a\in k$.
By these remarks we have that the isomorphism
$$
\Hom_k (\alpha_{p,k},\Gak)\times \Ext^1_k(\apk,\Gak)\too
\Ext^1_k(\apk,\apk),
$$
deduced from \ref{teo:ext(ap,ap)}, is given by
$$
(\beta,\gamma C_1)\mapsto [E_{\beta,\gamma}].
$$
The group scheme $E_{\beta,\gamma}$ is so defined:
$$
E_{\beta,\gamma}:=\Sp(k[T_1,T_2]/(T_1^p,T_2^p-\beta T_1))
$$

\begin{enumerate}
    \item  comultiplication
    \begin{align*}
    T_1\longmapsto &T_1\pt 1+1\pt T_1\\
    T_2\longmapsto &T_2\pt 1 + 1\pt T_2 + \gamma \frac{T_1^p\pt 1+ 1\pt T_1^p-(T_1\pt 1+1\pt T_1)^p}{p}
    \end{align*}
    \item counit
    \begin{align*}
    &T_1\longmapsto 0\\
    &T_2\longmapsto 0
    \end{align*}
    \item coinverse
    \begin{align*}
    &T_1\longmapsto -T_1\\
    &T_2\longmapsto -T_2
\end{align*}
or
 \begin{align*}
    &T_1\longmapsto -T_1\\
    &T_2\longmapsto -\gamma T_1^2-T_2
\end{align*}
if $p=2$.

\end{enumerate}
 In
\cite[4.3.1]{SS4} the following result  was proved.
\begin{prop}Let $\mu,\lb\in \pi R\setminus\{0\}$. Then $[\clE^{(\mu,\lb;E_{p}(\bbf{a},\mu,T))}_k]\in H_0^2({\Ga}_{,k},{\Gak})$
coincides with the class of
$$
\sum_{k=1}^{\infty}\frac{(\tilde{\fr}_{k-1})(\widetilde{\bbf{a}})}{\lb}
C_k,
$$
where $\fr-[\mu^{p-1}]=(\tilde{\fr}_0, \tilde{\fr}_1,\dots,
\tilde{\fr}_k,\dots)$ and  $\widetilde{\bbf{a}}\in \widehat{W}(R)$
is a lifting of $\bbf{a}\in \widehat{W}(R/\lb R)$.
\end{prop}
We deduce the following corollary about the extensions of
$\alpha_{p,k}$ by $\Gak$.
\begin{cor}\label{cor:ext(ap,ga)}Let $\mu,\lb\in \pi R\setminus\{0\}$. If $char(R)=0$ we assume $v(\lb_{(1)})>v(\mu)$.
Then $[\widetilde{\clE}^{(\mu,\lb;E_{p}(\bbf{a},\mu,T))}_k]\in H_0^2({\ap}_{,k},{\Ga}_{,k})$
coincides with the class of
$$
\frac{(\fr-{\mu^{p-1})}(\widetilde{a}_0)}{\lb}C_1,
$$
where $\widetilde{\bbf{a}}=(\widetilde{a}_0,\widetilde{a}_1),
\dots, \widetilde{a}_i,\dots) \in \widehat{W}(R)$ is a lifting of
$\bbf{a}\in \widehat{W}(R/\lb R)$.
\end{cor}
\begin{proof}
This follows from the fact that
$[{\clE}^{(\mu,\lb;E_{p}(\bbf{a},\mu,T))}_k]\mapsto
[\widetilde{\clE}^{(\mu,\lb;E_{p}(\bbf{a},\mu,T))}_k]$ through the
map
$$ \Ext^1_k(\Gak,\Gak)\simeq H_0^2(\Gak,\Gak)\too H_0^2(\apk,\Gak)\simeq
\Ext^1_k(\apk,\Gak).
 $$
\end{proof}

Let us take an extension class $[\clE^{(\mu,\lb;E_p(a,\mu;
T),j)}]$. Let $\widetilde{a}\in R$ be a lifting of $a\in R/\lb R$.
We have on the special fiber
$$
\clE^{(\mu,\lb;\widetilde{F}(T),j)}_k=\Sp(k[T_1,T_2]/(T_1^p,T_2^p-(-\frac{\widetilde{F}(T_1)^p(1+\mu
T_1)^{-j}-1}{\lb^p}))),
$$
where
$\widetilde{F}(T)=\sum_{i=0}^{p-1}\frac{\widetilde{a}(\widetilde{a}-\mu)\dots
(\widetilde{a}-(i-1)\mu)}{i!}T^i$. If $char(R)=p$ we have, more
precisely,
$$
\clE^{(\mu,\lb;\widetilde{F}(T),j)}_k=\Sp(k[T_1,T_2]/(T_1^p,T_2^p-(\frac{j\mu}{\lb^p}T_1))),
$$
with $j\mu \equiv 0 \mod \lb^p$.

Let us now suppose $char(R)=0$. Let us consider
$E_p([\widetilde{a}],\mu;T)\in R[[T]]$. We have, by
\ref{lem:suriettività mappa tra hom}(ii),
$E_p([\widetilde{a}],\mu;T)\equiv \widetilde{F}(T)\mod (\lb T,
\frac{(1+\mu T)^p-1)}{\mu^p})$. Thus, since
$$
(p-1)v(\lb)<v(p) \quad \text{ and } \quad T^p\equiv 0 \mod (\pi,
\frac{(1+\mu T)^p-1}{\mu^p}),
$$
we have
\begin{equation}\label{eq:E_p=F^p}
E_p([\widetilde{a}],\mu;T)^p\equiv \widetilde{F}(T)^p \mod (\lb^p
\pi, \frac{(1+\mu T)^p-1}{\mu^p}).
\end{equation}
We now suppose $p>2$. Let us  consider
$E_p([\frac{p}{\mu^{p-1}}\widetilde{a}^p]+V([\widetilde{a}^p]),\mu;T)
\in R[[T]]$. 
 By \eqref{eq:phi*:W(R)too W(R)}, we have
$E_p([\frac{p}{\mu^{p-1}}\widetilde{a}^p]+V([\widetilde{a}^p]),\mu;T)\equiv
E_p([\widetilde{a}^p],\mu^p;\frac{(1+\mu T)^p-1}{\mu^p})\mod
\lb^p$.
Moreover by definitions
$$
E_p([\frac{p}{\mu^{p-1}}\widetilde{a}^p]+V([\widetilde{a}^p]),\mu;T)\equiv
E_p([\frac{p}{\mu^{p-1}}\widetilde{a}^p],\mu;T)\equiv (1+\mu
T)^{\frac{p \widetilde{a}^p}{\mu^p}}\equiv
E_p([\widetilde{a}^p],\mu^p;\frac{(1+\mu T)^p-1}{\mu^p})\mod
T^p.
$$
Hence
\begin{equation*}
E_p([\frac{p}{\mu^{p-1}}\widetilde{a}^p]+V([\widetilde{a}^p]),\mu;T)\equiv
E_p([\widetilde{a}^p],\mu^p;\frac{(1+\mu T)^p-1}{\mu^p})\mod \lb^p
T^p.
\end{equation*}
Therefore,  since $T^p\equiv 0 \mod (\pi, \frac{(1+\mu
T)^p-1}{\mu^p})$,
\begin{equation}\label{eq:E_p=E_p mod lpTp}
E_p([\frac{p}{\mu^{p-1}}\widetilde{a}^p]+V([\widetilde{a}^p]),\mu;T)\equiv
E_p([\widetilde{a}^p],\mu^p;\frac{(1+\mu T)^p-1}{\mu^p})\mod
\bigg(\lb^p\pi,\frac{(1+\mu T)^p-1}{\mu^p}\bigg).
\end{equation}

 By \cite[2.9.1]{SS4} we have
$$
E_p([\widetilde{a}],\mu;T)^p(1+\mu
T)^{-j}E_p([\frac{p}{\mu^{p-1}}\widetilde{a}^p]+V([\widetilde{a^p}]),\mu;T)^{-1}=E_p(p[\widetilde{a}]-j[\mu]-[\frac{p}{\mu^{p-1}}\widetilde{a}^p]-V([\widetilde{a^p}]),\mu;T).
$$

 We remark that, by the definition of sum between Witt vectors  and
 the proof of \ref{prop:rad_p},
$$
p[\widetilde{a}]-j[\mu]-[\frac{p}{\mu^{p-1}}\widetilde{a}^p]-V([\widetilde{a}^p])=[pa-j\mu-\frac{p}{\mu^{p-1}}\widetilde{a}^p]+V(\bbf{c})
$$
with $\bbf{c}\equiv 0\mod \lb^p$.  Therefore
\begin{align*}
E_p(p[\widetilde{a}]-j[\mu]-[\frac{p}{\mu^{p-1}}\widetilde{a}^p]-V([\widetilde{a}^p]),\mu;T)&=E_p([p\widetilde{a}-j\mu-\frac{p}{\mu^{p-1}}\widetilde{a}^p],\mu;T)E_p
(\bbf{c},\mu^p;T^p) \\
&\equiv
E_p([p\widetilde{a}-j\mu-\frac{p}{\mu^{p-1}}\widetilde{a}^p],\mu;T)\mod
\bigg(\lb^p\pi, \frac{(1+\mu T)^p-1}{\mu^p}\bigg),
\end{align*}
where the last congruence follows from $T^p\equiv 0 \mod (\pi,
\frac{(1+\mu T)^p-1}{\mu^p})$ and $\bbf{c}\equiv 0\mod \lb^p$. The
case $p=2$ is similar using
$E_p([\frac{2}{\mu}\widetilde{a}^2]+V([\widetilde{a}^2])+\widetilde{V}([\widetilde{a}^2]),\mu;T)
\in R[[T]]$.

The above discussion implies, using \eqref{eq:E_p=F^p} and
\eqref{eq:E_p=E_p mod lpTp},
\begin{align*}
\frac{\widetilde{F}(T_1)^p(1+\mu
 T_1)^{-j}-1}{\lb^p}&\equiv\frac{E_p([p\widetilde{a}-j\mu-\frac{p}{\mu^{p-1}}\widetilde{a}^p],\mu;T_1)-1}{\lb^p}\\
 &\equiv \frac{p\widetilde{a}-j\mu-\frac{p}{\mu^{p-1}}\widetilde{a}^p}{\lb^p}T_1\mod
 \bigg(\pi, \frac{(1+\mu T_1)^p-1}{\mu^p}\bigg).
\end{align*}
On the other hand, with no restriction on the characteristic of
$R$, $\clE^{(\mu,\lb;E_{p}(aT),j)}_k\mapsto
\widetilde{\clE}^{(\mu,\lb;E_{p}(aT))}_k$ through the map
$\Ext^1_k(\apk,\apk)\to \Ext^1_k(\apk,\Gak)$.

Therefore $\clE^{(\mu,\lb;E_{p}(aT),j)}_k\simeq E_{\beta,\gamma}$
with
$\beta=(-\frac{p\widetilde{{a}}-j\mu-\frac{p}{\mu^{p-1}}\widetilde{{a}}^p}{\lb^p}
\mod \pi)$ and
$\gamma=(\frac{\widetilde{{a}}^p-\mu^{p-1}\widetilde{{a}}}{\lb}\mod
\pi)$.
So we have  proved 
the following result.
\begin{prop}
Let $\mu,\lb\in \pi R$. If $char(R)=0$ we also assume
$v(\lb_{(1)})> v(\lb),v(\mu)$. Then
$[\clE^{(\mu,\lb;E_{p}(a,\mu;T),j)}_k]\in \Ext_k^1({\apk},{\apk})$
coincides with the class of
$$
\bigg(-\frac{p\widetilde{{a}}-j\mu-\frac{p}{\mu^{p-1}}\widetilde{{a}}^p}{\lb^p},
\frac{\widetilde{{a}}^p-\mu^{p-1}\widetilde{a}}{\lb} C_1\bigg),
$$
if $char(R)=0$ and
$$
\bigg(\frac{j\mu}{\lb^p},
\frac{\widetilde{{a}}^p-\mu^{p-1}\widetilde{a}}{\lb} C_1\bigg),
$$
if $char(R)=p$, where $\widetilde{{a}}\in R$ is a lifting of $a\in
R/\lb R$.
\end{prop}
\subsection{Case char(R)=0 and  $\mathbf{v(\lb_{(1)})=v(\mu)> v(\lb)>0}$}

In this situation we have $${(\gmx{1})}_k\simeq \Z/p\Z \quad
\text{ and } \quad {(\glx{1})}_k\simeq \apk.$$

%
\begin{prop}
Let $\mu,\lb\in \piR$ be such that $ v(\mu)=v(\lb_{(1)})>v(\lb)$.
Then $\clE^{(\mu,\lb;E_{p}(aT),j)}_k$ is the trivial extension.
\end{prop}
\begin{proof}
We can suppose $\mu=\lb_{(1)}$. From  \ref{ex:Ext 1 con mu=lb_(1)}
it follows that any extension  of type
$[\clE^{(\lb_{(1)},\lb;F,j)}]$ is uniquely determined by the
induced extension  over $K$. Since by \ref{cor:p2 surjective} we
have $F(T)\equiv E_p(j\eta T)\mod \lb$ then
$[\clE^{(\lb_{(1)},\lb;F,j)}]$ is  the image of
$[\clE^{(\lb_{(1)},\lb_{(1)};E_p(j\eta T),j)}]$ through the
morphism
$$
\Ext^1(G_{\lb_{(1)},1},G_{\lb_{(1)},1})\too
\Ext^1(G_{\lb_{(1)},1},G_{\lb,1})
$$
induced by the map $\Z/p\Z\simeq G_{\lb_{(1)},1}\too G_{\lb,1}$
given by $T\mapsto \frac{\lb_{(1)}}{\lb}T$.  But the above
morphism is the zero morphism on the special fiber. So we are
done.
\end{proof}

  We remark that if $pv(\lb)\le v(\lb_{(1)})$ then $\eta\equiv
0\mod \lb$, indeed in such a case $v(\lb)\le
v(\lb_{(2)})=v(\eta)$.

\subsection{Case  char(R)=0 and $\mathbf{v(\lb_{(1)})=v(\lb)> v(\mu)>0}$}
 We have
$${(\gmx{1})}_k\simeq \ap \quad \text{ and }
{(\glx{1})}_k\simeq \Z/p\Z.$$ From \cite[III \S 6 7.3]{DG} it
follows that $\Ext_k^1(\alpha_{p,k},\Z/p\Z)=0$.
\subsection{Case  char(R)=0 and $\mathbf{v(\lb_{(1)})=v(\mu)= v(\lb)}$}

 We have
$${(\gmx{1})}_k\simeq \Z/p\Z \quad \text{ and }
{(\glx{1})}_k\simeq \Z/p\Z.$$ Without loss of generality we can
suppose $\mu=\lb=\lb_{(1)}$.
The Artin Schreier sequence $$0\to \Z/p\Z \too
\Gak\on{\fr-1}{\too} \Gak\too 0$$ induces the following exact
sequence
\begin{align*}
 \Hom_{k-gr}(\Z/p\Z,\Gak)\on{{\fr}-1}{\too}\Hom_{k-gr}(\Z/p\Z,\Gak)&\too \Ext_k^1(\Z/p\Z,\Z/p\Z)\too \\
\too\Ext_k^1&(\Z/p\Z,\Gak)\on{{\fr}-1}{\too}\Ext_k^1(\Z/p\Z,\Gak)
\end{align*}
There are canonical isomorphisms $\Hom_{k-gr}(\Z/p\Z,\Gak)\simeq
k$ and \mbox{$\Ext^1_k(\Z/p\Z,\Gak)\simeq k$}
(see \cite[III \S 6 4.3]{DG}). Therefore we have the exact
sequence
\begin{equation}\label{eq:ext(Z/pZ,Z/pZ)}
0\too \Z/p\Z\too \Ext_k^1(\Z/p\Z,\Z/p\Z)\too k/(F-1)k\too 0.
\end{equation}
We recall that
$\Ext^1_k(\Z/p\Z,\Gak)=H^2_0(\Z/p\Z,\Gak)$ is freely generated as
a right $k$-module by the class of the  cocycle
$C_1=\frac{U^p+V^p-(U+V)^p}{p}$. 
The above sequence splits and we have the following isomorphism,
$$
k/({\fr}-1)(k) \times \Z/p\Z\too \Ext^1_k(\Z/p\Z,\Z/p\Z),
$$
 given by
$$
(a,b)\mapsto [E_{a,b}].
$$
where the group scheme $E_{a,b}$ is so defined: let $\bar{a}\in k$
a lifting of $a$,
$$
E_{a,b}:=\Sp(k[T_1,T_2]/(T_1^p-T_1,T_2^p-T_2-\bar{a} T_1))
$$

\begin{enumerate}
    \item  comultiplication
    \begin{align*}
    T_1\longmapsto &T_1\pt 1+1\pt T_1\\
    T_2\longmapsto &T_2\pt 1 + 1\pt T_2 + b \frac{T_1^p\pt 1+ 1\pt T_1^p-(T_1\pt 1+1\pt T_1)^p}{p}
    \end{align*}
    \item counit
    \begin{align*}
    &T_1\longmapsto 0\\
    &T_2\longmapsto 0
    \end{align*}
    \item coinverse
    \begin{align*}
    &T_1\longmapsto -T_1\\
    &T_2\longmapsto -T_2
\end{align*}
or
    \begin{align*}
    &T_1\longmapsto -T_1\\
    &T_2\longmapsto -bT_1^2-T_2
\end{align*}
if $p=2$.
\end{enumerate}
  We now study the reduction on
the special fiber of the group scheme
$\clE^{(\lb_{(1)},\lb_{(1)};E_{p}(j \eta S),j)}$ with $j\in
\Z/p\Z$.
\begin{prop} For any $j\in \Z/p\Z$, $[\clE^{(\lb_{(1)},\lb_{(1)};E_{p}(j\eta T),j)}_k]=E_{0,j}\in
\Ext^1_k({\Z/p\Z},{\Z/p\Z})$. 

\end{prop}
\begin{proof}
As group schemes, $\clE^{(\lb_{(1)},\lb_{(1)};E_{p}(j \eta
T),j)}\simeq\Z/p^2\Z$, if $j\neq 0$, and
$\clE^{(\lb_{(1)},\lb_{(1)};1,0)}\simeq \Z/p\Z\times \Z/p\Z$
otherwise. In particular $\clE^{(\lb_{(1)},\lb_{(1)};E_{p}(j \eta
T),j)}_k$ has a  scheme-theoretic section. It is easy to see that
$\clE^{(\lb_{(1)},\lb_{(1)};E_{p}(j \eta T),j)}_k\simeq E_{0,b}$
with $$b=(-j\frac{\eta^p}{\lb_{(1)}(p-1)!} \mod \pi)=j,$$ since
$\frac{\eta^p}{\lb_{(1)}}\equiv
\frac{\lb_{(2)}^p}{\lb_{(1)}}\equiv 1\mod \pi$ and $(p-1)!\equiv
-1\mod \pi$ (Wilson Theorem). 
\end{proof}

We finally remarks that the exact sequence
\eqref{eq:ext(Z/pZ,Z/pZ)} also reads as
$$
0\too H_0^2(\Z/p\Z,\Z/p\Z)\too \ext^1_k(\Z/p\Z,\Z/p\Z)\too
H^1(k,\Z/p\Z),
$$
where the isomorphism $H^1(k,\Z/p\Z)\simeq k/(F-1)k$ comes from
the Artin-Schreier Theory. And we have the following commutative
diagram with exact rows

$$
\xymatrix{0\ar[r] & H_0^2(\mu_{p,K},\mu_{p,K})\ar[r]&
\Ext^1_S(\mu_{p,K},\mu_{p,K})\ar[r]&H^1(K,\Z/p\Z)\ar[r]&0\\
 0\ar[r] & H_0^2(\Z/p\Z,\Z/p\Z)\ar[r]\ar^{\wr}[d]\ar_{\wr}[u]&
\Ext^1_S(\Z/p\Z,\Z/p\Z)\ar[r]\ar[d]\ar[u]&H^1(S,\Z/p\Z)\ar[r]\ar[d]\ar[u]&0\\
 0\ar[r] & H_0^2(\Z/p\Z,\Z/p\Z)\ar[r]&
\Ext^1_k(\Z/p\Z,\Z/p\Z)\ar[r]&H^1(k,\Z/p\Z)\ar[r]&0}
$$
This diagram shows that the above proposition is an expression of
the theory unifying the Kummer and  Artin-Schreier theories.
\appendix

\section{} \label{A:Caruso}
%

\begin{center}
{\textsc{ Classification of integral models of $\mu_{p^2,K}$ via
Breuil-Kisin theory}} \\

\vspace{2mm}

\sc{by Xavier Caruso\footnote{\noindent {\sc IRMAR, Université de
Rennes 1, Campus de Beaulieu, 35042 Rennes Cedex, France}\\
\hspace*{.5cm} {\em Email address: xavier.caruso@normalesup.org}}}
\end{center}

\vspace{5mm}

In this appendix, we show how the theory presented by Breuil in
\cite{breuil} and developed by Kisin in \cite{kisin} gives us the
possibility to obtain very quickly in some cases a statement
analogous to \ref{cor:modelli di Z/p^2 Z sono cosi'2} of this
paper. Although our approach is certainly more efficient, it has
at least two defects. First, it forces us to assume $p > 2$ and
$R$ complete of unequal characteristic with perfect residue field.
Therefore, the situation that we will consider in this appendix is
slightly less general than the one discussed in the paper. Second,
we do not obtain an explicit description of models of
$\mu_{p^2,K}$, but instead we describe some objects of linear
algebra which correspond to these models through Breuil-Kisin
theory.

\bigskip

\noindent \textit{Acknowledgment.} The author thanks Eike Lau for
interesting remarks and discussions and for pointing out to him
the fact that Kisin's classification \emph{a priori} requires $p >
2$.

\subsection*{Statement of the main theorem}

Let us fix notation. Let $p$ be an odd prime number and $k$ a
perect field of characteristic $p$. We denote by $W = W(k)$ (resp.
$W_n = W_n(k)$) the ring of Witt vectors (resp. of truncated Witt
vectors) with coefficients in $k$ and $K_0$ its fraction field of
$W$. For any integer $n$, $W_n[[u]]$ is endowed with a continuous
(for the $u$-adic topology) rings endomorphism $\phi$ defined as
the usual Frobenius on $W_n$ and by $\phi(u)=u^p$. Let's fix a
totally ramified extension $K$ of $K_0$ of degree $e$ and an
uniformizer $\pi$ of $K$. We denote by $E(u)$ the minimal
polynomial of $\pi$ over $K_0$ and $R$ the ring of integers of
$K$. This one corresponds to the d.v.r. considered as base ring in
Tossici's paper.

\medskip

Let $\kis$ denote the following category:
\begin{itemize}
\item objects are $W_2[[u]]$-modules $\M$ with no $u$-torsion
endowed with a continuous (for the $u$-adic topology)
$\phi$-semi-linear endomorphism (called Frobenius) $\phi_\M : \M
\to \M$ whose image generates a sub-module containing $E(u) \M$;
\item morphisms are the $W_2[[u]]$-linear maps which commute with
Frobenius.
\end{itemize}

In \cite{kisin}, Kisin has constructed an anti-equivalence of
categories between $\kis$ and the category of finite, flat and
commutative $R$-group schemes annihilated by $p^2$. If we compose
with the Cartier duality we obtain an equivalence of category.
For our aims, an important property of the latter equivalence will
be the following: if $\M$ is the object of $\kis$ associated to a
group scheme $G$, then $\M[1/u]$ completely determines the Galois
representation $G(\bar K)$ (where $\bar K$ is an algebraic closure
of $K$), \emph{i.e.} the generic fiber of $G$. From this fact, it
is easy to prove that $G$ is a model of $\mu_{p^2,K}$ if and only
if $\M[1/u]$ is isomorphic to $W_2((u))$ endowed with the usual
Frobenius. We are going to prove the following result, which is
the exact analogue in our context of \ref{cor:modelli di Z/p^2 Z
sono cosi'2}.

\begin{thm}
\label{theo:principal} Let $\M$ be the object $\kis$ associated to
a finite flat $R$-group scheme whose generic fiber is isomorphic
to $\mu_{p^2,K}$. Then, there exist $n,m\in \N$, $a \in k[[u]]$
satisfying $\frac e {p-1} \geq m \geq n \geq 0$ and
\begin{eqnarray}
\phi(a) & \equiv & 0 \hphantom{F(u) u^m} \pmod {u^n} \label{eq:cond1} \\
u^{e-m(p-1)} \phi(a) - u^e a & \equiv & F(u) u^m \hphantom{0}
\pmod {u^{pn}} \label{eq:cond2}
\end{eqnarray}
together with two elements $e_1$ and $e_2$ in $\M$ such that:
\begin{itemize}
\item[i)] $\M$ is generated over $W_2[[u]]$ by $e_1$ and $e_2$
with the unique relation $u^{m-n} e_1 = p e_2$; \item[ii)]
Frobenius is given by $\phi(e_1) = u^{n(p-1)} e_1$ and $\phi(e_2)
= u^{m(p-1)} e_2 + \big[ u^{-n} \phi(a) - u^{m(p-1)-n} a \big]
e_1$.
\end{itemize}

Furthermore two triples $(n,m,a)$ have equal reduction $(n,m,(a
\pmod {u^n}))$ if and only if the associated groups are
isomorphic.

Conversely, any triple $(n,m,a)$ satisfying \eqref{eq:cond1} and
\eqref{eq:cond2} comes from a finite flat $R$-group scheme whose
generic fiber is isomorphic to $\mu_{p^2,K}$.
\end{thm}

The last assertion of the Theorem is easy: one just need to check
that the $\phi$-module $\M$ defined by conditions i) and ii) is
actually an object of $\kis$. From now on, we concentrate
ourselves to the proof of the rest of the Theorem.

\subsection*{Proof of existence}

Let $\M$ be a $\phi$-module over $W_2[[u]]$ such that $\M[1/u]$ is
isomorphic to $W_2((u))$ endowed with the usual Frobenius. Let us
denote by $\M_1$ the kernel of the multiplication by $p$ on $\M$
and $\M_2 = \M / \M_1$. It is easy to verify that they are both
modules over $W_1[[u]] = k[[u]]$ with no $u$-torsion. Moreover
they inherit endomorphisms $\phi_{\M_1}$ and $\phi_{\M_2}$ whose
images still generate a module which contains $E(u) \M_1 = u^e
\M_1$ and $E(u) \M_2 = u^e \M_2$ respectively. In the following,
we will write $\phi$ for $\phi_\M$, $\phi_{\M_1}$ and
$\phi_{\M_2}$.

\begin{lem}
The module $\M_1$ is free of rank $1$ over $k[[u]]$. Moreover,
there  exists a base $(e_1)$ of $\M_1$ and an integer $n \in [0,
\frac e {p-1}]$ such that $\phi(e_1) = u^{n(p-1)} e_1$.
\end{lem}

\begin{proof}
Since $k[[u]]$ is a discrete valuation ring, the fact that $\M_1$
has no $u$-torsion implies that its freeness. Moreover, it is
certainly of rank $1$ because $\M_1[1/u]$ is isomorphic to the
kernel of the multiplication by $p$ on $W_2((u))$, that is
$k((u))$. Let $x$ be any basis of $\M_1$. From above it follows
that we can consider it as an element of $k((u))$. We can write $x
= u^n y$ where $y$ is invertible in $k[[u]]$. Then, if we set $e_1
= u^n$, it is a basis of $\M_1$ and we have $\phi(e_1) = u^{np} =
u^{n(p-1)} e_1$ as expected.
\end{proof}

In the same way it is possible to prove that $\M_2 = k[[u]] \bar
e_2$ with $\phi(\bar e_2) = u^{m(p-1)} \bar e_2$ for some integer
$m \in [0, \frac e {p-1}]$. Let $e_2 \in \M$ be any lifting of
$\bar e_2$. Clearly it is a generator of $\M[1/u]$ as
$W_2((u))$-module. We deduce that
\begin{equation}
\label{eq:e1e2} e_1 = p u^{-\delta} \alpha e_2
\end{equation}
where $\delta$ is an integer and $\alpha$ is invertible in
$W_2[[u]]$. In fact, $\alpha$ is defined modulo $p W_2[[u]]$, so
that we may (and will) consider it as an element of $k[[u]]$. The
fact that $e_1$ generates $\M_1$ easily implies $\delta \geq 0$.
Moreover, since $\phi(e_2) \equiv u^{m(p-1)} e_2 \pmod p$,
applying $\phi$ to \eqref{eq:e1e2}, we obtain
$$\phi(e_1) = p u^{-p \delta} \phi(\alpha) \phi(e_2) = p u^{m(p-1) - p
\delta} \phi(\alpha) e_2.$$ Therefore $\phi(e_1) =
u^{(m-\delta)(p-1)} \frac{\phi(\alpha)} \alpha e_1$. Comparing
with $\phi(e_1) = u^{n(p-1)} e_1$, we obtain $m - \delta = n$ and
$\phi(\alpha) = \alpha$. The first condition gives $\delta = m -
n$ (and in particular $m \geq n$), while the second one implies
$\alpha \in \F_p^\star$. So, up to replacing $e_1$ by $\frac {e_1}
\alpha$, we may assume $\alpha = 1$.

\medskip

We have just proved that $\M$ is generated by two vectors $e_1$
and $e_2$ related by \eqref{eq:e1e2} with $\alpha = 1$. This is
exactly what appears in the statement of Theorem
\ref{theo:principal}. We also know that $\phi(e_1) = u^{n(p-1)}
e_1$. It still remains to precise the shape of $\phi(e_2)$. Let
$z$ denote the image of $e_2 \in \M[1/u]$ through the isomorphism
$\M[1/u] \simeq W_2((u))$. From $\phi(\bar e_2) = u^{m(p-1)} \bar
e_2$, we deduce that, up to multiplying $e_2$ by a $(p-1)$-th root
of unity, we can write $z = u^m + p a$, with $a \in k((u))$. After
some calculations, we obtain
$$\phi(e_2) = u^{m(p-1)} e_2 + \big[ u^{-n} \phi(a) - u^{m(p-1)-n} a
\big] e_1 = u^{m(p-1)} e_2 + b e_1.$$ Hence RHS have to be in
$\M$, which gives directly \eqref{eq:cond1} (using $m \geq n$).
Now, using $E(u) \M \subset \left< \phi(e_1), \phi(e_2) \right>$
(where the notation $\left< \, \cdots \right>$ means the generated
submodule), we find that there exist $x,y \in W_2[[u]]$ such that
$$E(u) e_2 = x u^{n(p-1)} e_1 + y (u^{m(p-1)} e_2 + b e_1).$$
Reducing modulo $p$, we have $y=u^{e-m(p-1)} + p y'$ for some $y'
\in W_2[[u]]$. Since $x$ and $y'$ are defined modulo $p$, one may
consider them as elements of $k[[u]]$. After some calculations, we
get $F(u) = b u^{n-pm+e} + x u^{pn-m} + y' u^{(p-1)m}$ where
$F(u)$ is defined by the equality $E(u) = u^e + p F(u)$. As $m
\geq n$, we have $(p-1)m \geq pn -m$. This shows that the equality
we obtained is equivalent to the congruence $b u^{n-pm-e} \equiv
F(u) \pmod {u^{pn-m}}$. Replacing $b$ by its expression, we
finally obtain \eqref{eq:cond2}. It remains to be prove that $a$
is an element of $k[[u]]$ (\emph{a priori}, we only know that it
belongs to $k((u))$), which is clear from\eqref{eq:cond1}.

\subsection*{End of the proof}

Let $\M$ and $\M'$ be two $\phi$-modules presented as in Theorem
\ref{theo:principal} with parameters $(n,m,a)$ and $(n',m',a')$
respectively. We want to prove that $\M$ and $\M'$ are isomorphic
if and only if $n = n'$, $m = m'$ and $a \equiv a' \pmod {u^n}$.
Let us assume that there exists an isomorphism $f : \M \to \M'$.
Since $\phi$ acts by multiplication by $u^{n(p-1)}$ (resp.
$u^{n'(p-1)}$) on $\M_1 = \ker p_{|\M}$ (resp. $\M'_1 = \ker
p_{|\M'}$), we get $n = n'$. In fact, examining the action of
$\phi$ on $e_1$ and $e'_1$ it is easy to see that there exists
$\alpha \in \F_p^\star$ such that $f(e_1) = \alpha e'_1$. In the
same way, regarding actions of $\phi$ on quotients $\M/\M_1$ and
$\M'/\M'_1$ , we have $m = m'$. Then, equalities $u^{m-n} e_1 = p
e_2$ and $u^{m-n} e'_1 = p e'_2$ give $f(e_2) = \alpha e'_2 + x
e_1$ for some element $x \in k[[u]]$. A little computation shows
that the compatibility with $\phi$ implies
$$\alpha u^{-n} \phi(a') - \alpha u^{m(p-1)-n} a' + \phi(x) u^{n(p-1)} =
x u^{m(p-1)} + \alpha u^{-n} \phi(a) - \alpha u^{m(p-1)-n} a,$$
which gives $\phi(t) = u^{m(p-1)} t$ where we set $t = \alpha (a'
- a) + u^n x$. Comparing $u$-adic valuations of both sides, we see
that any solution $t$ has to be divisible by $u^m$. As $m \geq n$,
we have $a \equiv a' \pmod {u^n}$ as wanted. Conversely, if $a
\equiv a' \pmod {u^n}$, it is sufficient to set $\alpha = 1$ and
$x = \frac{a - a'} {u^n}$ to obtain an isomorphism $f : \M \to
\M'$.

\subsection*{Discussion about $\mathbf{p=2}$}

In this appendix, we have assumed $p > 2$. The only restriction
for that is the fact that Kisin's equivalence of categories
between $\kis$ and finite flat and commutative $R$-group schemes
annihilated by $p^2$ requires (at least nowadays) $p \neq 2$.

Nevertheless, in \cite{kisin2}, Kisin proved that the equivalence
of categories still holds for $p=2$ between   $R$-group schemes
with unipotent special fiber and a suitable subcategory of $\kis$.
Using this result and the method of this appendix, one can
classify all  models of $\mu_{p^2,K}$ with unipotent special
fiber, which  we will call, for simplicity, \emph{unipotent}
models. Actually, it is not so hard to find by hand non-unipotent
models. Indeed, by Lemma 3.1 of thee paper, one knows that one
model $G$ of $\mu_{p^2,K}$ must be an extension of $G_{\mu,1}$ by
$G_{\lambda,1}$. If $G$ is not unipotent, it means that
$G_{\mu,1}$ or $G_{\lambda,1}$ is itself not unipotent, that is
$\mu$ or $\lambda$ is a unit. By Lemma 3.2 of the paper, it
follows that $\lambda$ is a unit, {\emph i.e.} $G_{\lambda,1}
\simeq \mu_{p,R}$. This remaining case can be treated rather
easily (\emph{cf.} \S 3.4 of \emph{loc. cit.}): in particular it
does not use Sekiguchi-Suwa's theory.

Moreover in a very recent note (still unpublished), Eike Lau
proved that Kisin's classification also works in general for $p =
2$. Consequently, the method of the appendix can finally be
applied without restriction on the characteristic $p$.

\bibliographystyle{smfplain}
\bibliography{bibliomodel3} \nocite{*}

\end{document}